\documentclass[12pt]{amsart}
\usepackage{amsmath,amssymb,amsthm,color, euscript, enumerate,comment,amsfonts}
\usepackage{dsfont}
\usepackage{longtable}

\newcommand{\lcm}{\mathrm{lcm}}

\renewcommand{\l}{\left}
\renewcommand{\r}{\right}

\newcommand{\maru}[1]{{\ooalign{\hfil#1\/\hfil\crcr
\raise.167ex\hbox{\mathhexbox20D}}}}

\newcommand{\ruby}[2]{%
 \leavevmode
 \setbox0=\hbox{#1}%
 \setbox1=\hbox{\tiny #2}%
 \ifdim\wd0>\wd1 \dimen0=\wd0 \end{lemma}se \dimen0=\wd1 \fi
 \hbox{%
   \kanjiskip=0pt plus 2fil
   \xkanjiskip=0pt plus 2fil
   \vbox{%
     \hbox to \dimen0{%
       \tiny \hfil#2\hfil}%
     \nointerlineskip
     \hbox to \dimen0{\mathstrut\hfil#1\hfil}}}}

\newcommand{\la}{\langle}
\newcommand{\ra}{\rangle}

\newcommand{\Z}{\mathbb{Z}}
\newcommand{\C}{\mathbb{C}}
\newcommand{\R}{\mathbb{R}}

\newcommand{\Q}{\mathbb{Q}}

\newcommand{\g}{\mathfrak{g}}

\newcommand{\h}{\mathfrak{h}}

\newcommand{\Sym}{{\rm Sym}}
\newcommand{\Alt}{{\rm Alt}}

\newcommand{\rank}{\mathrm{rank}\,}

\newcommand{\End}{\mathrm{End}}

\newcommand{\aut}{\mathrm{Aut}\,}
\newcommand{\Aut}{\mathrm{Aut}\,}

\renewcommand{\hom}{\mathrm{Hom}}

\newcommand{\GO}{\mathrm{GO}}

\newcommand{\PSO}{\mathrm{PSO}}

\newcommand{\AGL}{\mathrm{AGL}}
\newcommand{\Dih}{\mathrm{Dih}}

\newcommand{\Span}{\mathrm{Span}}

\newcommand{\vacuum}{\mathrm{1\hspace{-3.2pt}l}}
\newcommand{\vac}{\vacuum}

\newcommand{\irr}{\mathrm{Irr}}

\makeatletter \@addtoreset{equation}{section}

\theoremstyle{plain}
\newtheorem{theorem}{Theorem}[section]
\newtheorem{proposition}[theorem]{Proposition}
\newtheorem{lemma}[theorem]{Lemma}

\theoremstyle{definition}
\newtheorem{definition}[theorem]{Definition}

\theoremstyle{remark}
\newtheorem{remark}[theorem]{Remark}
\numberwithin{equation}{section}

\title[Automorphism groups of cyclic orbifold VOAs]{Automorphism groups of cyclic orbifold vertex operator algebras associated with the Leech lattice and some non-prime isometries}

\author{Koichi Betsumiya}
  \address[K. Betsumiya] {Graduate School of Science and Technology,
Hirosaki University, Hirosaki 036-8561, Japan} 
  \email{betsumi@hirosaki-u.ac.jp}
\author{Ching Hung Lam} %
  \address[C. H. Lam] {Institute of Mathematics, Academia Sinica, Taipei 10617, Taiwan} 
  \email{chlam@math.sinica.edu.tw}
\author[H. Shimakura]{Hiroki Shimakura}%
\address[H. Shimakura]{Graduate School of Information Sciences, Tohoku University, Sendai, 980-8579, Japan }%
\email {shimakura@tohoku.ac.jp}%
\date{}
\thanks{C.H. Lam was partially supported by a research grant AS-IA-107-M02 of Academia Sinica  and MOST grant  107-2115-M-001-003-MY3  of Taiwan}
\thanks{H. Shimakura was partially supported by JSPS KAKENHI Grant Numbers JP19KK0065 and JP20K03505.}

\newcommand{\sfr}[2]{\leavevmode\kern-.1em
  \raise.5ex\hbox{\the\scriptfont0 #1}\kern-.1em
  /\kern-.15em\lower.25ex\hbox{\the\scriptfont0 #2}}

\begin{document}

\begin{abstract}We determine the automorphism groups of the cyclic orbifold vertex 
operator algebras associated with coinvariant lattices of isometries of the Leech lattice in the conjugacy classes $4C,6E,6G,8E$ and $10F$.
As a consequence, we have determined the automorphism groups of all the $10$ vertex operator algebras in \cite{Ho2}, which are useful to analyze holomorphic vertex operator algebras of central charge $24$. 
\end{abstract}

\maketitle


\section{Introduction}

The classification of holomorphic vertex operator algebras (VOAs) of central charge $24$ has been completed except for the uniqueness of the moonshine VOA, and  several uniform proofs are proposed (see \cite{Ho2,MS,HM,ELMS,CLM} and the references therein).
One of them, proposed by H\"ohn in \cite{Ho2}, is to view a holomorphic VOA $V$ of central charge $24$ with $V_1\neq0$ as a simple current extension of the tensor product VOA $V_{L_{\mathfrak {g}}} \otimes V_{\Lambda_g}^{\hat{g}}$.
Here $V_{L_{\mathfrak{g}}}$ is the lattice VOA associated with a lattice $L_\mathfrak{g}$ related to the root lattice of $\g=V_1$, $g$ is some isometry of the Leech lattice $\Lambda$ and $V_{\Lambda_g}^{\hat{g}}$ is the cyclic orbifold VOA associated with the coinvariant lattice $\Lambda_g$ and a lift $\hat{g}\in\Aut(V_{\Lambda_g})$ of the restriction of $g$ to $\Lambda_g$.
H\"ohn also described the possible isometries $g$; there are $10$ different cases \cite[Table 4]{Ho2}. 

These $10$ VOAs $V_{\Lambda_g}^{\hat{g}}$ would be useful to analyze the holomorphic VOAs of central charge $24$ .
One fundamental question is to determine their full automorphism groups.
When $|g|$ is a prime, i.e, the conjugacy class of $g$ is $2A$, $2C$, $3B$, $5B$ or $7B$, the full automorphism group of $V_{\Lambda_{\mathfrak{g}}}^{\hat{g}}$ has been determined in \cite{Gr98,Sh04,CLS, Lam18b, LamCFT} as in Table \ref{Table:main}.
There are still $5$ remaining non-prime cases and the full automorphism groups of $V_{\Lambda_g}^{\hat{g}}$ have not been determined yet for the conjugacy classes $4C, 6E,6G,8E$ and $10F$.

Recently, a criterion for the existence of extra automorphisms in $\Aut(V_L^{\hat{g}})$, i.e, not induced from $\Aut(V_L)$,  is given in \cite{LS21}.
One of the results in \cite{LS21} is that $V_L^{\hat{g}}$ has an automorphism which sends an irreducible $V_L^{\hat{g}}$-submodule of $V_L$ to an irreducible $V_L^{\hat{g}}$-module of twisted type if and only if $L$ is a coinvariant sublattice $\Lambda_g$ of certain isometry $g$ of the Leech lattice $\Lambda$ under the assumptions that $|g|$ is prime and $L$ has no roots.
Hence it would be interesting to classify non-prime order isometries $g$ of $\Lambda$ such that $V_{\Lambda_g}^{\hat{g}}$ have the same property; in fact, the property holds if $g$ belongs to the $5$ conjugacy classes $4C, 6E,6G,8E$ and $10F$ (see Remark \ref{R:twist}).

In this article, we determine the full automorphism groups of $V_{\Lambda_g}^{\hat{g}}$ for these $5$ cases as in Table \ref{Table:main} (see \cite{ATLAS} for the notations of groups). 

\begin{longtable}{|c|c|c|c|}
\caption{Automorphism groups $\Aut(V_{\Lambda_g}^{\hat{g}})$}\label{Table:main}
\\ \hline 
Class& $\irr(V_{\Lambda_g}^{\hat{g}})$&$\Aut(V_{\Lambda_g}^{\hat{g}})$ & Reference \\ \hline
$2A$&$2^{10}$&$\GO_{10}^+(2)$&\cite{Gr98,Sh04}\\
$2C$&$2^{10}4^2$&$2^{11}.(2^{10}.(\Sym_{12}.\Sym_3$))&\cite{Sh04}\\
$3B$&$3^8$&$\Omega_8^-(3).2$&\cite{CLS}\\
$4C$& $2^24^6$&$2^{21}.\GO_7(2)$& Theorem \ref{T:main} \\
$5B$&$5^6$&$\Omega_6^+(5).2$&\cite{Lam18b}\\
$6E$&$2^63^6$&$\GO_6^+(2)\times \GO_6^+(3)$& Theorem \ref{T:main}\\ 
$6G$&$2^44^23^5$&$ 2^{5+4}.(\Sym_3 \times \Sym_3 \times \Sym_3) \times \Omega_5(3).2.$& Theorem \ref{T:main}\\
$7B$& $7^5$&$\Omega_5(7).2$&\cite{LamCFT}\\
$8E$& $2.4.8^4$&$2^{11+9}.\Sym_6$& Theorem \ref{T:main} \\
$10F$&$2^24^25^4$&$(Q_8:2^2).(2\times \Sym_3)\times \GO_4^+(5)$& Theorem \ref{T:main}\\
\hline 
\end{longtable}

Our method is similar to that in \cite{Sh04,Lam18b,LamCFT}. The main idea is to study the orbit
of the irreducible $V_{\Lambda_g}^{\hat{g}}$-submodule $V_{\Lambda_g}(1)$ of $V_{\Lambda_g}$ under the conjugate actions of $\Aut(V_{\Lambda_{g}}^{\hat{g}})$ (cf. Definition \ref{Mconj} and \eqref{Eq:conjact}), where  $V_{\Lambda_g}(j)=\{ v\in V_{\Lambda_g}\mid \hat{g}(v)= e^{2\pi \sqrt{-1} j/n} v\}$ and $0\leq j\leq |\hat{g}|-1$. 

Let $g\in O(\Lambda)$ be an isometry of the Leech lattice $\Lambda$. 
We assume that the conjugacy class of $g$ is $4C, 6E,6G,8E$ or $10F$.
Then $g=id$ on $\Lambda_g^*/\Lambda_g$ and by \cite{La19}, all irreducible $V_{\Lambda_{g}}^{\hat{g}}$-modules are simple current modules. It implies that the set $\irr(V_L^{\hat{g}})$ of  
  all (isomorphism classes of) irreducible $V_L^{\hat{g}}$-modules  forms a (finite) abelian group under the fusion product. In this case, $(\irr(V_L^{\hat{g}}), q)$ has a non-degenerate quadratic space structure  \cite[Theorem 3.4]{EMS} with $q$ defined by 
\begin{equation*}
q:\irr(V_L^{\hat{g}})\to \Q/\Z,\quad M\mapsto \varepsilon(M)\pmod\Z,
\end{equation*}
where $\varepsilon(M)$ is the conformal weight of $M$.
Since the action of $\aut(V_{\Lambda_g}^{\hat{g}})$ on $\irr(V_{\Lambda_g}^{\hat{g}})$ by conjugation preserves the conformal weights and the fusion products, it also preserves the quadratic space structure.
Hence there exists a natural group homomorphism $\mu: \aut(V_{\Lambda_g}^{\hat{g}}) \to O(\mathrm{Irr}(V_{\Lambda_g}^{\hat{g}}))$, where $O(\mathrm{Irr}(V_{\Lambda_g}^{\hat{g}}))$ is the orthogonal group of the quadratic space $(\mathrm{Irr}(V_{\Lambda_g}^{\hat{g}}), q)$. For the cases that we are interested in, it turns out that the natural map $\mu$ is injective. Therefore, we may consider $\aut(V_{\Lambda_g}^{\hat{g}})$ as a subgroup of $O(\mathrm{Irr}(V_{\Lambda_g}^{\hat{g}}))$. 

Since $V_{\Lambda_g}^{\hat{g}}$ is the fixed-point subspace of $\hat{g}$ on $V_{\Lambda_g}$, the centralizer $C_{\Aut(V_{\Lambda_g})}(\hat{g})$ acts on $V_{\Lambda_g}^{\hat{g}}$ by restriction with the kernel $\langle \hat{g}\rangle$.
The group $C_{\Aut(V_{\Lambda_g})}(\hat{g})/ \langle \hat{g}\rangle$ is also the pointwise stabilizer of $\{ V_{\Lambda_g}(i) \mid 0 \leq i\leq |\hat{g}|-1\}\   ( \subset \irr( V_{\Lambda_g}^{\hat{g}}))$ (cf. Proposition \ref{P:stabVL1}).
Note that $C_{\Aut(V_{\Lambda_g})}(\hat{g})$ can be computed since $\Aut(V_{\Lambda_g})$ is described in \cite{DN}.

First, to determine the automorphism group, we have to know if $\aut(V_{\Lambda_g}^{\hat{g}})$ has {extra} automorphisms.
A general criteria that $\aut(V_{L}^{\hat{g}})$ has extra automorphisms, along with an explicit construction, is given in \cite{LS21}.
It turns out that the 
coinvariant lattice $\Lambda_g$ satisfies this criteria when $g\in 4C, 6E,6G,8E$ or $10F$ (see Proposition \ref{P:Nc}) and $V_{\Lambda_g}^{\hat{g}}$ has extra automorphisms. 
In addition, we describe their conjugate actions on some elements of $\irr( V_{\Lambda_g}^{\hat{g}})$.

Next, we will determine the orbit
of $V_{\Lambda_g}(1)$ under the conjugate actions of $\Aut(V_{\Lambda_{g}}^{\hat{g}})$. 
We consider a certain subset $\mathcal{S}_g$ of $\irr(V_{\Lambda_{g}}^{\hat{g}})$ which consists of singular elements satisfying the same properties as $V_{\Lambda_g}(1)$ (see \eqref{Eq:Sg} and \eqref{Eq:Sg2}).  By using the extra automorphisms defined in \cite{LS21}, we prove that  $\Aut(V_{\Lambda_{g}}^{\hat{g}})$ acts transitively on $S_g$ (see Propositions \ref{P:tranSg} and \ref{6GS}). Therefore, $\Aut(V_{\Lambda_{g}}^{\hat{g}})$ is a small index subgroup of the full orthogonal group.  By using the structure of the full orthogonal group and computer calculations by MAGMA (\cite{MAGMA}), we determine the shape of $\Aut(V_{\Lambda_{g}}^{\hat{g}})$. 

The organization of the article is as follows. In Section 2, we recall some basic notions and properties about integral lattices and lattice VOAs. We also recall some facts about the automorphism groups of lattice VOAs. In Section 3, we review the structures of irreducible $V_L^{\hat{g}}$-modules and determine the conjugates of irreducible $V_L^{\hat{g}}$-modules by some elements of $C_{\Aut(V_L)}( \hat{g})$. 
In Section \ref{S:extra}, we will recall from \cite{LS21} a general criteria 
for which the orbifold VOA $V_L^{\hat{g}}$ contains some extra automorphisms. In Section 5, we study the coinvariant lattices $\Lambda_g$ for $g\in 4C,6E,6G,8E$ and $10F$ and discuss some of their properties. In Section 6, 
we show that $\Aut(V_{\Lambda_{g}}^{\hat{g}})$ acts transitively on  certain subsets of $\irr(V_{\Lambda_g}^{\hat{g}})$ and determine the automorphism group 
$\Aut(V_{\Lambda_{g}}^{\hat{g}})$ for each case.

\section{Preliminary} 
We first recall some notions and properties about integral lattices and lattice vertex operator algebras. 

\subsection{Lattices}\label{S:lattice}

Let $L$ be a lattice with the positive-definite bilinear form $(\cdot | \cdot )$. We 
denote its isometry group by $O(L)$. 
The dual lattice $L^*$ of $L$ is defined to be the lattice
\[
L^*=\{ \alpha\in \Q\otimes_\Z L\mid ( \alpha| L)\subset \Z \}.
\]

Assume that $L$ is even, equivalently, $(\alpha|\alpha)\in2\Z$ for all $\alpha\in L$.
Then $L$ is integral, that is, $L< L^*$, and the discriminant group $\mathcal{D}(L)$ is defined to 
be the quotient group $L^*/L$, which is a 
finite abelian group. 
Let   
\begin{equation}\label{Eq:qL}
q:\mathcal{D}(L)\to \Q/\Z,\quad \alpha+L\mapsto \frac{(\alpha|\alpha)}{2}+\Z. 
\end{equation} 
Then $q$ defines a quadratic form on $\mathcal{D}(L)$.
The associated bilinear form on $\mathcal{D}(L)$ given by $(\alpha+L|\beta+L)=(\alpha|\beta)+\Z$ is non-degenerate, and hence $(\mathcal{D}(L),q)$ is a non-degenerate quadratic space.

Let $$O(\mathcal{D}(L),q)=\{f\in \Aut(\mathcal{D}(L))\mid q(f(\alpha+L))=q(\alpha+L)\ \text{for all } \alpha+L\in \mathcal{D}(L)\}$$ 
be the orthogonal group of $(\mathcal{D}(L),q)$.
Then $O(L)$ acts on $\mathcal{D}(L)$ as a subgroup of $O(\mathcal{D}(L),q)$ but this action is not necessarily faithful.
We often denote $O(\mathcal{D}(L),q)$ by $O(\mathcal{D}(L))$ for short.

The \emph{(square) norm} of a vector $v\in \Q\otimes_\Z L$ is defined to be the value $(v|v)$.
For $U\subset \Q\otimes_\Z L$ and $s\in \Q$, we define
\[
  U(s)=\{ \alpha\in U\mid (\alpha|\alpha) =s\}
\]
to be the set of all norm $s$ vectors in $U$.
We also define the minimum norm of $U$ by $$\min(U)=\min\{(\alpha|\alpha)\mid \alpha \in U \}.$$
We call a vector of norm $2$ in an even lattice a 
\emph{root}. 
Note that the set of roots of an even lattice forms a root system.

Let $g\in O(L)$ of order $n$.
The fixed-point sublattice $L^g$ of $g$ and the \emph{coinvariant lattice} $L_g$ of $g$ are defined to be 
\begin{equation}
L^g=\{\alpha\in L\mid g\alpha=\alpha\}\quad \text{ and } \quad L_g = \{ \alpha\in L \mid (\alpha| \beta) =0 \text{ for all } \beta\in L^g\}.\label{Eq:coinv}
\end{equation}
Clearly $$\rank L=\rank L^g+\rank L_g.$$
Notice that the restriction of $g$ to $L_g$ is a fixed-point free isometry of order $n$; we also denote it by $g$ without confusion.

\subsection{VOAs, modules and automorphisms}

A \emph{vertex operator algebra} (VOA) $(V,Y,\vac,\omega)$ is a $\Z$-graded vector space $V=\bigoplus_{i\in\Z}V_i$ over the complex field $\C$ equipped with a linear map
$$Y(a,z)=\sum_{i\in\Z}a_{(i)}z^{-i-1}\in ({\rm End}\ V)[[z,z^{-1}]],\quad a\in V,$$
the \emph{vacuum vector} $\vac\in V_0$ and the \emph{conformal vector} $\omega\in V_2$
satisfying certain axioms (\cite{Bo,FLM}). 
The operators $L(m)=\omega_{(m+1)}$, $m\in \Z$, satisfy the Virasoro relation:
$$[L{(m)},L{(n)}]=(m-n)L{(m+n)}+\frac{1}{12}(m^3-m)\delta_{m+n,0}c\ {\rm id}_V,$$
where $c\in\C$ is called the \emph{central charge} of $V$.

For a VOA $V$, a $V$-\emph{module} $(M,Y_M)$ is a $\C$-graded vector space $M=\bigoplus_{i\in\C} M_{i}$ equipped with a linear map
$$Y_M(a,z)=\sum_{i\in\Z}a_{(i)}z^{-i-1}\in (\End\ M)[[z,z^{-1}]],\quad a\in V$$
satisfying a number of conditions (\cite{FHL,DLM2}).
We often denote it by $M$.
If $M$ is irreducible, then there exists $\varepsilon(M)\in\C$ such that $M=\bigoplus_{m\in\Z_{\geq 0}}M_{\varepsilon(M)+m}$ and $M_{\varepsilon(M)}\neq0$; the number $\varepsilon(M)$ is called the \emph{conformal weight} of $M$.
Let $\irr(V)$ denote the set of all isomorphism classes of irreducible $V$-modules.
We often identify an element in $\irr(V)$ with its representative.
Assume that the fusion products $\boxtimes$ are defined on $\irr(V)$ (\cite{HL}).
An irreducible $V$-module $M^1$ is called a \emph{simple current module} if for any irreducible $V$-module $M^2$, the fusion product $M^1\boxtimes M^2$ is also an irreducible $V$-module.

Let $V$ be a VOA.
A linear automorphism $\tau$ of $V$ is called a (VOA) \emph{automorphism} of $V$ if $$ \tau\omega=\omega\quad {\rm and}\quad \tau Y(v,z)=Y(\tau v,z)\tau\quad \text{ for all } v\in V.$$
We denote the group of all automorphisms of $V$ by $\Aut(V)$. 

\begin{definition}\label{Mconj} (\cite{DLM2})
Let $V$ be a VOA and $\tau\in\Aut(V)$. For $g\in\Aut(V)$, let $M=(M,Y_M)$ be a $g$-twisted $V$-module.
The \emph{$\tau$-conjugate} $(M\circ \tau, Y_{M\circ \tau} (\cdot, z))$ of $M$ is defined as follows:
\begin{equation}
\begin{split}
& M\circ \tau =M \quad \text{ as a vector space;}\\
& Y_{M\circ \tau} (a, z) = Y_M(\tau a, z)\quad \text{ for any } a\in V.
\end{split}\label{Eq:conjact}
\end{equation}
Then $(M\circ \tau, Y_{M\circ \tau} (\cdot, z))$ defines a $\tau^{-1} g\tau$-twisted $V$-module.
\end{definition}

The $\tau$-conjugation defines an $\Aut(V)$-action on the isomorphism classes of twisted $V$-modules.  
It also does an $\Aut(V)$-action on $\irr(V)$.
Since this conjugation preserves the characters and the fusion products of $V$-modules, we obtain the following lemma.
\begin{lemma}\label{Lem:conj} Let $M,M^1,M^2$ be $V$-modules and let $\tau\in\Aut(V)$.
\begin{enumerate}[{\rm (1)}]
\item If $M$ is irreducible, then so is $M\circ \tau$.
In addition, both $M$ and $M\circ \tau$ have the same conformal weight.
\item Assume that the fusion product is defined on $V$-modules.
Then $(M^1\circ \tau)\boxtimes (M^2\circ \tau)\cong (M^1\boxtimes M^2)\circ \tau$.
\end{enumerate}
\end{lemma}

\subsection{Lattice VOAs and their automorphism groups}
In this section, we recall a few facts about lattice VOAs and their automorphism groups from \cite{FLM,DN, LY2}.

Let $L$ be a positive-definite even lattice and let $(\cdot |\cdot )$ be the positive-definite symmetric bilinear form on $\R\otimes_\Z L$.
Let $M(1)$ be the Heisenberg VOA associated with $\mathfrak{h}=\C\otimes_\Z L$. 
We  extended the form $(\cdot|\cdot)$  $\C$-bilinearly to $\mathfrak{h}$. 
Let $\C\{L\}=\bigoplus_{\alpha\in L}\C e^\alpha$ be the twisted group algebra such that $e^\alpha e^\beta=(-1)^{(\alpha|\beta)}e^{\beta}e^{\alpha}$, for  $\alpha,\beta\in L$.
The lattice VOA $V_L$ associated with $L$ is defined to be $M(1) \otimes \C\{L\}$. 

Let $\hat{L}$ be a central extension of $L$ associated with the commutator map $(\cdot|\cdot)\pmod2$ on $L$.
Let $\Aut(\hat{L})$ be the automorphism group of $\hat{L}$.
For $\varphi\in \Aut (\hat{L})$, we define the element $\bar{\varphi}\in\Aut(L)$ by $\varphi(e^\alpha)\in\{\pm e^{\bar\varphi(\alpha)}\}$, $\alpha\in L$.
Set $O(\hat{L})=\{\varphi\in\Aut(\hat L)\mid \bar\varphi\in O(L)\}.$
For $x\in\h$, we set 
\begin{equation}
\sigma_x=\exp(-2\pi\sqrt{-1}x_{(0)})\in\Aut(V_L).\label{Eq:sigma}
\end{equation}
Then $\{\varphi\in O(\hat{L})\mid \bar{\varphi}=id\}=\{\sigma_x\mid x\in ((1/2)L^*)/L^*\}.$
By \cite[Proposition 5.4.1]{FLM}, we have an exact sequence
\begin{equation}
  1 \to \{\sigma_x\mid x\in ((1/2)L^*)/L^*\} \to  O(\hat{L})\stackrel{-}\rightarrow  O(L) \to  1.\label{Eq:exact}
\end{equation}
Note that $\aut (V_L) = N(V_L)\,O(\hat{L})$ (\cite{DN}), where $N(V_L)=\l\la \exp(a_{(0)}) \mid a\in (V_L)_1 \r\ra$.
Note also that if $L(2)=\emptyset$, then $N(V_L)=\{\sigma_x\mid x\in\h/L^*\}$
and \[
  1 \to \{\sigma_x\mid x\in \h/L\} \to  \Aut(V_L)\stackrel{-}\rightarrow  O(L) \to  1
\] is an exact sequence.

An element $\phi\in O(\hat{L})$ is called a \emph{standard lift} of $g\in O(L)$ if $\bar{\phi}=g$ and $\phi(e^\alpha)=e^\alpha$ for $\alpha\in L^g$.
The following lemma can be found in \cite{EMS,LS19}.
\begin{lemma}\label{L:standardlift}
Let $g\in O(L)$ and let $\hat{g}\in O(\hat{L})$ be a standard lift of $g$.
\begin{enumerate}[{\rm (1)}]
\item All standard lifts of $g$ are conjugate by some elements in $\{\sigma_x\mid x\in\h\}$. 
\item If $g$ is fixed-point free or has odd order, then $|\hat{g}|=|g|$.
\end{enumerate}
\end{lemma}

Let $g\in O(L)$ be a fixed-point free isometry of order $n$ and let $\hat{g}\in O(\hat{L})$ be a standard lift of $g$.
By the proof of \cite[Theorem 5.15]{LY2} (cf.\ \cite[Lemma 2.2]{LS21}), we obtain
\begin{equation}
\{\sigma_y\mid y\in ((1-g)^{-1}L^*)/L^*\}=\{\sigma_y\mid y\in \h/L^*\}\cap C_{\Aut(V_L)}(\hat{g}),
\label{Eq:hom}
\end{equation}
where $C_{\Aut(V_L)}(\hat{g})$ is the centralizer of $\hat{g}$ in $\Aut(V_L)$.
In addition, we obtain
\begin{equation}
\hom(L/(1-g)L, \Z/n\Z)\cong\{\sigma_y\mid y\in ((1-g)^{-1}L^*)/L^*\}.\label{Eq:hom2}
\end{equation}

The following has been proved in \cite[Theorem 
5.15]{LY2} when $n$ is prime; for non-prime $n$,
we can prove it by the same argument:
 \begin{proposition}
 \label{normalizer}
Let $L$ be an even positive-definite lattice with $L(2)=\emptyset$.
Let $g$ be a fixed-point free isometry of $L$ of order $n$ and
$\hat{g}$ a standard lift of $g$ in $O(\hat{L})$ of order $n$.
Then we have the following exact sequence:
\begin{align*}
\begin{split}
   1\longrightarrow \hom(L/(1-g)L, \Z/n\Z)
  \longrightarrow C_{\aut(V_L)}(\hat{g})
  \stackrel{-}\longrightarrow C_{O(L)}(g) \longrightarrow 1.
\end{split}
\end{align*}
\end{proposition}

\section{Irreducible $V_L^{\hat{g}}$-modules and $\hat{g}$-conjugates}\label{S:irred}
Let $L$ be a positive-definite even lattice.
Let $g$ be a fixed-point free isometry of $L$ of order $n$ and let $\hat{g}\in O(\hat{L})$ be a standard lift of $g$.
Then the order of $\hat{g}$ is $n$.
Let $V_L^{\hat{g}}$ be the fixed-point subspace of $\hat{g}$, which is a subVOA.
Recall that $V_L^{\hat{g}}$ is strongly regular, namely, rational, $C_2$-cofinite, self-dual and of CFT-type \cite{CM,Mi}.

In this section, we review the construction of irreducible $V_L^{\hat{g}}$-modules and determine their conjugates by some elements in 
$\hom(L/(1-g)L, \Z/n\Z)$; see \eqref{Eq:hom} and \eqref{Eq:hom2}.
Throughout this section, we assume that $g=id$ on $\mathcal{D}(L)$, that is, 
\begin{equation}
(1-g)L^*\subset L.\label{Eq:Assump}
\end{equation}

\subsection{Submodules of irreducible (untwisted) $V_L$-modules}\label{S:un}
In this subsection, we deal with the irreducible $V_L^{\hat{g}}$-modules which appear as $V_L^{\hat{g}}$-submodules of irreducible (untwisted) $V_L$-modules.

Let $\lambda+L\in \mathcal{D}(L)$ and $V_{\lambda+L}=M(1)\otimes\Span_\C\{e^\alpha\mid \alpha\in\lambda+L\}$.
Then $V_{\lambda+L}$ has an irreducible $V_L$-module structure (\cite{FLM}).
It follows from $(1-g)L^*\subset L$ that $V_{\lambda+L}$ is $\hat{g}$-invariant. 
Fix a $\hat{g}$-module isomorphism $\hat{g}_{\lambda+L}$ of $V_{\lambda+L}$ of order $n$. 
For $0\le i\le n-1$, set 
$$V_{\lambda+L}(i)=\{v\in V_{\lambda+L}\mid \hat{g}_{\lambda+L}(v)=\exp(2\pi\sqrt{-1}i/n)v\},$$ which is an irreducible $V_L^{\hat{g}}$-module.
For $\lambda+L\in \mathcal{D}(L)\setminus \{L\}$, $1\le i\le n-1$ and $0\le j\le n-1$, it is easy to see that 
\begin{equation}
\varepsilon(V_L(0))=0,\quad \varepsilon(V_L(i))=1,\quad \varepsilon(V_{\lambda+L}(j))=\frac{1}{2}\min(\lambda+L).\label{Eq:cwun}
\end{equation}

Let $h\in N_{O(L)}(g)$ and let $\hat{h}\in N_{\Aut(V_L)}(\hat{g})$ be a lift of $h$ (cf.\ Proposition \ref{normalizer}).
Then $\hat{h}(V_L^{\hat{g}})=V_L^{\hat{g}}$.
In addition, for $\lambda+L\in\mathcal{D}(L)$, we have $V_{\lambda+L}\circ\hat{h}\cong V_{h^{-1}(\lambda)+L}$ as $V_L$-modules.
We regard $\hat{h}$ as an element of $\Aut(V_L^{\hat{g}})$ by restriction.
Then we obtain the following:

\begin{lemma}\label{Lem:conjlift}
Let $h\in N_{O(L)}(g)$ and let $\hat{h}\in N_{\Aut(V_L)}(\hat{g})$ be a lift of $h$.
As sets of isomorphism classes of irreducible $V_L^{\hat{g}}$-modules,  we have 
\[
\{V_{\lambda+L}(j)\circ\hat{h}\mid 0\le j\le n-1\}= \{V_{h^{-1}(\lambda)+L}(j)\mid 0\le j\le n-1\}.
\]
\end{lemma}

The $\sigma_x$-conjugate of $V_{\lambda+L}(i)$ for some $x\in\h$ is described in \cite[Lemma 3.2]{LS21}. 
\begin{lemma}\label{Lem:conjhom} Let $\alpha,\lambda \in L^*$ and $0\le i,s\le n-1$.
If $(\alpha|\lambda)\in s/n+\Z$, then as $V_L^{\hat{g}}$-modules, $$V_{\lambda+L}(i)\circ \sigma_{(1-g)^{-1}\alpha}\cong V_{\lambda+L}(i-s).$$ 
\end{lemma}

\begin{lemma}\label{Re:conj}
Assume that $(1-g)L^*=L$.
Then $\hom(L/(1-g)L, \Z/n\Z)$ acts faithfully on $\{V_{\lambda+L}(i)\mid \lambda+L\in \mathcal{D}(L),\ 0\le i\le n-1\}$.
\end{lemma}
\begin{proof} Let $\sigma_{(1-g)^{-1}\alpha}\in \hom(L/(1-g)L, \Z/n\Z)$, $\alpha\in L^*$, as in \eqref{Eq:hom2}.
Assume that $V_{\lambda+L}(i)\circ \sigma_{(1-g)^{-1}\alpha}\cong V_{\lambda+L}(i)$ for all $0\le i\le n-1$ and all $\lambda+L\in\mathcal{D}(L)$.
By Lemma \ref{Lem:conjhom}, $(\alpha|\lambda)\in\Z$ for all $\lambda\in L^*$, that is, $\alpha\in L$.
It follows from the assumption $(1-g)^{-1}L=L^*$ that
 $(1-g)^{-1}\alpha\in L^*$, which shows $\sigma_{(1-g)^{-1}\alpha}=id$. 
\end{proof}

Recall that $V_L$ is a $\Z_n$-graded simple current extension of $V_L^{\hat{g}}$ and the set of all irreducible $V_L^{\hat{g}}$-submodules of $V_L$ is $\{V_L(i)\mid 0\le i\le n-1\}$, which forms a cyclic group of order $n$ under the fusion products (\cite{DJX}).
Hence, by \cite{Sh04}, we obtain the following:
\begin{proposition}[{\cite[Theorem 3.3]{Sh04}}] \label{P:stabVL1} \ 
\begin{enumerate}[{\rm (1)}]
\item The stabilizer of $V_L(1)$ in $\Aut(V_L^{\hat{g}})$ is $C_{\Aut(V_L)}(\hat{g})/\langle \hat{g}\rangle$.
\item The stabilizer of $\{V_L(i)\mid 0\le i\le n-1\}$ in $\Aut(V_L^{\hat{g}})$ is $N_{\Aut(V_L)}(\langle\hat{g}\rangle)/\langle\hat{g}\rangle$.
\end{enumerate}
\end{proposition}

\subsection{Submodules of irreducible twisted $V_L$-modules}\label{S:tw}

In this subsection, we deal with the irreducible $V_L^{\hat{g}}$-modules which appear as $V_L^{\hat{g}}$-submodules of irreducible $\hat{g}^i$-twisted $V_L$-modules for some $1\le i\le n-1$.
We adopt the notation in \cite{La19} and use the descriptions of the irreducible $\hat{g}^i$-twisted $V_L$-modules in \cite[Sections 3.2 and 3.3]{AbeLY} (see also \cite{Le,DL,BK04}).
Let $1\le i\le n-1$ and set $f=g^i$.
Then the order of $f$ is $m=n/\gcd(n,i)$. 
Let $\hat{f}$ be a standard lift of $f$.
Then the order of $\hat{f}$ is either $m$ or $2m$.
Note that $f$ is not necessarily fixed-point free on $\R\otimes_\Z L$ and that $\hat{g}^i$ is not necessarily a standard lift of $f=g^i$.

By the assumption that $(1-g)L^*\subset L$, for any $\lambda+L\in \mathcal{D}(L)$, we have $f(\lambda)+L=\lambda+L$.
Let $P_0^f:\h\to \C\otimes_\Z L^f$ be the orthogonal projection.
Then $V_L$ has exactly $|\mathcal{D}(L)|$ irreducible $\hat{f}$-twisted $V_L$-modules up to isomorphisms \cite{DLM2}; they are given by $$V_{\lambda+L}[\hat{f}]=M(1)[f]\otimes\C[P_0^f(\lambda+L)]\otimes T_\lambda,\quad \lambda+L\in \mathcal{D}(L)$$
where $M(1)[f]$ is the ``$f$-twisted" free bosonic space, $\C[P_0^f(\lambda+L)]$ is a module over the group algebra of $P_0^f(L)$ and $T_\lambda$ is an irreducible module over a certain ``$f$-twisted" central extension of $L_f$ labeled by the characters of $(1-f)L^*/(1-f)L$ (\cite{Le,DL}).
Note that the labeling of $T_\lambda$ depends on the choice of certain normal subgroup of the ``$f$-twisted" central extension of $L_f$.

\begin{remark}\label{Re:Tlambda} When $m=n$, the set of linear characters of $(1-f)L^*/(1-f)L$ is $$\{\exp(-2\pi\sqrt{-1}(\cdot|(1-f)^{-1}\lambda))\mid \lambda+L\in \mathcal{D}(L)\}$$
and we choose the labeling so that $T_\lambda$ is associated with $\exp(-2\pi\sqrt{-1}(\cdot|(1-f)^{-1}\lambda))$.
\end{remark}

Since both $\hat{f}$ and $\hat{g}^i$ are lifts of $g^i$, there exists $v\in\h$ such that $\hat{g}^i=\sigma_v\hat{f}$ (see \eqref{Eq:exact}).
Up to conjugation by an element in $\Aut(V_L)$, we may assume that  $v\in P_0^f(\h)$ (cf.\ \cite[Lemma 4.5]{LS19}).
Then $\hat{f}$ and $\sigma_v$ are commutative, and the order of $\sigma_v$ divides $|\hat{f}|$, which shows $v\in(1/|\hat{f}|)L^*$.
Modifying $V_{\lambda+L}[\hat{f}]$ by the $\Delta$-operator in \cite{Li} associated with $\sigma_v$, we obtain an irreducible $\sigma_v\hat{f}$-twisted-module $V_{\lambda+L}[\hat{g}^i]$;
as a vector space
$$V_{\lambda+L}[\hat{g}^i]=M(1)[f]\otimes\C[v+P_0^f(\lambda+L)]\otimes T_\lambda,$$
and its conformal weight is given by 
\begin{equation}
\varepsilon(V_{\lambda+L}[\hat{g}^i])=\frac{1}{2}\min(v+P_0^f(\lambda+L))+\frac{1}{4m^2}\sum_{j=1}^{m-1}j(m-j)\dim\h_{(j)},\label{Eq:cw}
\end{equation}
where $\h_{(j)}$ is the eigenspace of $f$ with eigenvalue $\exp(2\pi\sqrt{-1}j/m)$.

\begin{remark}\label{R:intwist} Let $v\in (1/s)L^*$.
Then $\sigma_v\in\Aut(V_L)$ satisfies $\sigma_v^s=id$.
Applying the $\Delta$-operator associated with $\sigma_v$ to the irreducible (untwisted) $V_L$-module $V_{\lambda+L}$, we obtain the $\sigma_v$-twisted $V_L$-module structure on $V_{v+\lambda+L}$.
Conversely, any irreducible $\sigma_v$-twisted $V_L$-module is obtained by this construction. (\cite{Li, DLM2})
\end{remark}

For $1\le i\le n-1$, the irreducible $\hat{g}^i$-twisted $V_L$-module $V_{\lambda+L}[\hat{g}^i]$ is $\hat{g}$-stable (\cite[Section 4]{DL}).
Let $\{V_{\lambda+L}[\hat{g}^i](j)\mid 0\le j\le n-1\}$ be the set of (inequivalent) irreducible $V_L^{\hat{g}}$-submodules of $V_{\lambda+L}[\hat{g}^i]$.
Later, we will specify the labeling by using their conformal weights and the fusion products as in \cite{La19}.
By Remark \ref{Re:Tlambda}, we obtain the following lemma.

\begin{lemma}\label{Lem:conjun} Let $\alpha\in L^*$ and let $1\le i\le n-1$ with $\gcd(n,i)=1$.
Then, as $V_L^{\hat{g}}$-modules, $$V_{L}[\hat{g}^i]\circ \sigma_{(1-g^i)^{-1}\alpha}\cong V_{\alpha+L}[\hat{g}^i].$$
\end{lemma}

\subsection{Irreducible $V_L^{\hat{g}}$-modules and the quadratic form}
In this subsection, we summarize the classification of irreducible $V_L^{\hat{g}}$-modules and review the quadratic form on $\mathrm{Irr}(V_L^{\hat{g}})$. 

It was proved in \cite{DRX} that any irreducible  $V_{L}^{\hat{g}}$-module is a submodule of an irreducible $\hat{g}^i$-twisted $V_{L}$-module for some $0\le i\le n-1$.
Here we identify $V_{\lambda+L}(j)$ with $V_{\lambda+L}[\hat{g}^0](j)$.
By Sections \ref{S:un} and \ref{S:tw} and \cite[Main Theorem 1]{La19}, we obtain the following:

\begin{theorem} Let $L$ be a positive-definite even lattice and let $g$ be an order $n$ fixed-point free isometry of $L$.
Assume that $(1-g)L^*\subset L$.
Then 
$$\irr(V_L^{\hat{g}})=\{V_{\lambda+L}[\hat{g}^i](j)\mid \lambda+L\in \mathcal{D}(L),\ 0\le i,j\le n-1\}$$
In particular, $V_L^{\hat{g}}$ has exactly $|\mathcal{D}(L)|n^2$ inequivalent irreducible modules.
Moreover, any irreducible $V_L^{\hat{g}}$-module is a simple current module.
\end{theorem}

By the theorem above, $\irr(V_L^{\hat{g}})$ forms a (finite) abelian group under the fusion product.
Recall that for $M\in\irr(V_L^{\hat{g}})$, the conformal weight $\varepsilon(M)$ is rational.
Let 
\begin{equation}
q:\irr(V_L^{\hat{g}})\to \Q/\Z,\quad M\mapsto \varepsilon(M)\pmod\Z.\label{Eq:qVOA}
\end{equation}
Then the form $\langle\ \mid\ \rangle:\mathrm{Irr}(V_L^{\hat{g}})\times \mathrm{Irr}(V_L^{\hat{g}})\to\Q/\Z$ associated with $q$ is defined by 
\begin{equation}
\langle M^1| M^2\rangle=q(M^1\boxtimes M^2)-q(M^1)-q(M^2)\pmod\Z.\label{Eq:bilinear}
\end{equation}

\begin{theorem}\label{Thm:EMS} {\rm (\cite[Theorem 3.4]{EMS})} $(\mathrm{Irr}(V_L^{\hat{g}}),q)$ is a non-degenerate quadratic space, that is, the form $\langle\ \mid\ \rangle$ is non-degenerate and bilinear.
\end{theorem}

The quadratic space structure of $\mathrm{Irr}(V_L^{\hat{g}})$ can be determined by using an even unimodular lattice containing $L$ as in \cite[Theorems 5.3 and 5.8]{La19}; in Section \ref{S:6}, we will explicitly describe these structure when $L$ is the coinvariant lattice associated with some isometry of the Leech lattice.

\section{Extra automorphisms of $V_L^{\hat{g}}$}\label{S:extra}

In this section, we recall from \cite[Sections 4 and 5]{LS21} some extra automorphisms of $V_L^{\hat{g}}$, not in $N_{\Aut(V_L)}(\langle\hat{g}\rangle)/\langle\hat{g}\rangle$, for some $L$ and $g$.

Let $t\in\Z_{>0}$ and let $k_i\in\Z_{\ge2}$, $1\le i\le t$.
Let $n$ be the least common multiple of $k_1,\dots,k_t$.
Let $R_i$ be the root lattice of type $A_{k_i-1}$ and let $R=\bigoplus_{i=1}^tR_i$ be the orthogonal sum of $R_i$.
Fix a base $\Delta_i=\{\alpha_1^i,\dots,\alpha_{k_i-1}^i\}$ of the root system $R_i(2)$ of type $A_{k_i-1}$.
Let $\lambda_1^i,\dots,\lambda_{k_i-1}^i$ be the fundamental weights of $R_i(2)$ with respect to $\alpha_1^i,\dots,\alpha_{k_i-1}^i$.
Then $\Delta=\bigcup_{i=1}^t\Delta_i$ is a base of $R(2)$.
Let $\rho_{\Delta_i}$ be the Weyl vector of $R_i$ and set $$\chi_\Delta=(\frac{\rho_{\Delta_1}}{k_1},\dots,\frac{\rho_{\Delta_t}}{k_t})\in\Q\otimes_\Z R.$$
For $1\le i\le t$ and $1\le j\le k_i-1$, we have $(\rho_{\Delta_i}|\alpha^i_j)=1$.
Then $\{(\chi_\Delta|v)\mid v\in R\}=(1/n)\Z$.
Fix $\gamma\in R$ satisfying 
\begin{equation}
(\gamma|\chi_\Delta)\in \frac1n+\Z.\label{Eq:gamma}
\end{equation}

Let $\nu:R^*\to R^*/R\cong\bigoplus_{i=1}^t\Z_{k_i}$. 
Let $C$ be a subgroup of $R^*/R$ such that the lattice $N=L_A(C)=\nu^{-1}(C)$ is even.
Set 
\begin{equation}
L=L_B(C)=\{v\in N\mid (v|\chi_\Delta)\in\Z\}.\label{Eq:LBC}
\end{equation}

Let $g_{\Delta_i}$ be the Coxeter element of the Weyl group of $R_{i}$ associated with $\Delta_i$, i.e., $g_{\Delta_i}=r_{\alpha_1^i}\cdots r_{\alpha_{k_i-1}^i}$ is a product of simple reflections $r_{\alpha_j^i}$.
Let $e=(e_i)\in\bigoplus_{i=1}^t\Z_{k_i}^\times\subset R^*/R$.
Then $g_{\Delta_i}^{e_i}$ is a fixed-point free isometry of $R_{i}$ of order $k_i$ since $\gcd(e_i,k_i)=1$.
Hence 
$$g_{\Delta,e}=(g_{\Delta_1}^{e_1},\dots,g_{\Delta_t}^{e_t})\in O(N)$$
is a fixed-point free isometry of order $n$.
By \cite[Section 4]{LS21}, we have
\begin{equation}
g_{\Delta,e}(\chi_\Delta)\in\chi_\Delta-\lambda_e+R,\label{Eq:gchi}
\end{equation}
where $\lambda_e=(\lambda_{e_1}^1,\dots,\lambda_{e_t}^t)\in R^*$.

We now assume the following (cf. \cite[Lemmas 4.9 and 4.10]{LS21}):
\begin{enumerate}[(i)]
\item $\chi_\Delta\in (1/n)N^*$;
\item $\lambda_e\in L$, equivalently, $(\lambda_e|\lambda_e)\in2\Z$ (\cite[Lemma 4.8]{LS21}) .
\end{enumerate}
By Assumption (i), we have $|N:L|=|L^*:N^*|=n$.
It follows from $\chi_\Delta\in L^*$, \eqref{Eq:gamma} and \eqref{Eq:LBC} that 
\begin{equation}
N/L=\langle \gamma+L\rangle,\quad L^*/N^*=\langle \chi_\Delta+N^*\rangle.\label{Eq:L^*/N^*}
\end{equation}
Since $g_{\Delta,e}$ belongs to the Weyl group of $N$, we have $g_{\Delta,e}(\lambda+N)=\lambda+N$ for $\lambda+N\in \mathcal{D}(N)$.
By Assumption (ii), we have $\lambda_e\in N^*$ and $g_{\Delta,e}\in O(L)$.
Hence $g_{\Delta,e}(\lambda+L)=\lambda+L$ for $\lambda+L\in N^*/L$.
By \eqref{Eq:gchi} and (ii), $g_{\Delta,e}(\chi_\Delta+L)=\chi_\Delta+L$.
By \eqref{Eq:L^*/N^*}, we have $g_{\Delta,e}(\lambda+L)=\lambda+L$ for all $\lambda+L\in \mathcal{D}(L)$, equivalently, 
\begin{equation}
(1-g_{\Delta,e})L^*\subset L.
\end{equation}

Let $\hat{g}\in\Aut(V_N)$ be a (standard) lift of $g_{\Delta,e}$ and let 
$$\varphi=\sigma_{-\chi_\Delta}
=\exp(2\pi\sqrt{-1}(\chi_\Delta)_{(0)})\in \Aut(V_N).$$
Then $|\varphi|=n$ by Assumption (i), and 
$\varphi$ commutes with $\hat{g}$ 
by Assumption (ii).
By \cite[Section 5.3]{LS21}, there exists $ \zeta\in\Aut(V_N)$ such that
\begin{equation}
\zeta^{-1}\varphi\zeta=\hat{g}^{-1}\quad\text{and}\quad \zeta^{-1}\hat{g}\zeta=\varphi\quad\text{on}\quad V_N.\label{Eq:relations}
\end{equation}
Hence $V_N^{\langle\hat{g},\varphi\rangle}=V_L^{\hat{g}}$ and $\zeta\in\Aut(V_L^{\hat{g}})$.
Thus we obtain the following:

\begin{theorem}[{\cite[Theorem 5.3]{LS21}}] \label{T:extra} Let $N=L_A(C)$ and $L=L_B(C)$ be the even lattices as above.
Let $\hat{g}\in\Aut(V_L)$ be a (standard) lift of $g_{\Delta,e}$.
Then there exists an automorphism $\zeta$ of $V_{L}^{\hat{g}}$ such that $V_{L}\circ\zeta\cong V_{N}^{\hat{g}}$.
In particular, $\{V_{j\gamma+L}(0)\circ\zeta\mid 0\le j\le n-1\}=\{V_L(j)\mid 0\le j\le n-1\}$ as subsets of $\irr(V_L^{\hat{g}})$.
\end{theorem}

\begin{remark} By Proposition \ref{P:stabVL1} (2), the automorphism $\zeta$ in Theorem \ref{T:extra} does not belong to $N_{\Aut(V_L)}(\langle\hat{g}\rangle)/\langle\hat{g}\rangle$; such an automorphism is called an extra automorphism in \cite{LS21}.
\end{remark}

Next, we discuss the $\zeta$-conjugates of some irreducible $V_L^{\hat{g}}$-modules.
\begin{lemma}\label{phiandgtwisted}
For any $1\le i\le n-1$ and $\lambda+N\in \mathcal{D}(N)$, $V_{\lambda+N}[\hat{g}^i]\circ \zeta$ is an irreducible $\varphi^i$-twisted $V_N$-module.
In particular, there exists $\mu \in L^*\setminus N^*$ such that 
as $\varphi^i$-twisted $V_N$-modules,
\[
V_{\lambda+N}[\hat{g}^i]\circ \zeta \cong V_{\mu+N}.
\]

\end{lemma}
\begin{proof}
By \eqref{Eq:relations}, the $\zeta$-conjugate of the irreducible  
$\hat{g}^i$-twisted $V_N$-module is an irreducible $\varphi^i$-twisted $V_N$-module (see Definition \ref{Mconj}). 
Since $\varphi^i=\sigma_{-i\chi_{\Delta}}$ and $i\chi_{\Delta}\in L^*\setminus N^*$ if $i\notin n\Z$, we obtain the latter assertion from Remark \ref{R:intwist}.
\end{proof}

\begin{proposition}\label{P:extra} 
Let $0\le k\le n-1$ and set $\tilde{g}_k= \varphi^k \hat{g}$. 
Set $v_k= (1-g)^{-1} (k\chi_{\Delta})$ and $h_k=\sigma_{v_k} \zeta\in\Aut(V_N)$.
\begin{enumerate}[{\rm (1)}]
\item $h_k^{-1} \tilde{g}_k h_k =\varphi$ and $h_k^{-1}\varphi h_k=\hat{g}^{-1}$.
\item For $0\le i,j\le n-1$, as $V_{L}^{\hat{g}}$-modules, $$V_{j\gamma+L}(i)\cong V_{i\gamma+L}(ik-j)\circ h_k.$$
\end{enumerate}
\end{proposition}
\begin{proof}
(1) By \cite[Lemma 4.5]{LS19}, $\tilde{g}_k=\varphi^k \hat{g} = \sigma_{v_k} \hat{g} \sigma_{v_k}^{-1}$. It then follows from \eqref{Eq:relations} that 
$h_k^{-1} \tilde{g}_k h_k =\varphi $.
Since 
$\sigma_{v_k}$ commutes with $\varphi$, 
we obtain $h_k^{-1}\varphi h_k=\hat{g}^{-1}$ by \eqref{Eq:relations}.

(2) Consider the actions of $\varphi$ and $\hat{g}$ on $V_N$.
Then $V_{p\gamma+L}(q)$ is the common eigenspace of $\varphi$ and $\hat{g}$ with eigenvalues $\exp(2\pi\sqrt{-1}(-p)/n)$ and $\exp(2\pi\sqrt{-1}q/n)$, respectively.
By (1), we have $h_k^{-1}\hat{g}h_k=\hat{g}^{k}\varphi$ and $h_k^{-1}\varphi h_k=\hat{g}^{-1}$.
Hence
$h_k(V_{j\gamma+L}(i))$ is the common eigenspace of $\varphi$ and $\hat{g}$ with eigenvalues $\exp(2\pi\sqrt{-1}(-i/n))$ and $\exp(2\pi\sqrt{-1}({ik-j})/n)$, respectively.
Hence $h_k(V_{j\gamma+L}(i))=V_{i\gamma+L}(ik-j)$, 
which proves the assertion.
\end{proof}

\begin{proposition}\label{untwisted}
Under the action of $\Aut(V_{L}^{\hat{g}})$, any irreducible $V_{L}^{\hat{g}}$-module is conjugate to $V_{\mu+L}(\ell)$  for some $\mu+L\in\mathcal{D}(L)$ and $0\le \ell\le n-1$.
\end{proposition}

\begin{proof}
Let $V_{\lambda+L}[\hat{g}^i](q)$ be an irreducible $V_{L}^{\hat{g}}$-module.
Here $\lambda+L\in \mathcal{D}(L)$ and $0\le i,q\le n-1$.
Clearly, if $i=0$, then there is nothing to prove; we assume that $i\neq0$.

First, we assume that $\lambda\in N^*$.
By Lemma \ref{phiandgtwisted}, there exists $\mu\in L^*\setminus N^*$ such that as $\varphi^i$-twisted $V_N$-modules, 
\begin{equation}
V_{\lambda+N}[\hat{g}^i] \circ \zeta \cong V_{\mu+N}.\label{Eq:c1}
\end{equation}
Recall that 
$V_{\lambda+N}[\hat{g}^i] = \bigoplus_{j=0}^{n-1} V_{\lambda+j\gamma +L}[\hat{g}^i]
$
and 
$V_{\mu+N}= \bigoplus_{j=0}^{n-1} V_{\mu+j\gamma+L}  $
are direct sums of irreducible $\hat{g}^i$-twisted $V_{L}$-modules and irreducible $V_{L}$-modules, respectively.
Since $(1-g)L^*\subset L$, the $V_{L}$-module $V_{\mu+j\gamma+L}$ is $\hat{g}$-stable and thus, as $V_{L}^{\hat{g}}$-modules,
\begin{equation}
V_{\lambda+N}[\hat{g}^i] = \bigoplus_{j=0}^{n-1}\bigoplus_{k=0}^{n-1} V_{\lambda+j\gamma +L}[\hat{g}^i](k)\quad \text{and}\quad
V_{\mu +N } = \bigoplus_{j=0}^{n-1} \bigoplus_{k=0}^{n-1}V_{\mu+j\gamma+L} (k).\label{Eq:c2}
\end{equation}
Comparing \eqref{Eq:c1} with \eqref{Eq:c2}, we obtain as $V_{L}^{\hat{g}}$-modules, 
$$V_{\lambda+L}[\hat{g}^i](q)\circ\zeta\cong V_{\mu'+L}(\ell)$$ for some $\mu' \in L^*$ and for some $0 \leq \ell \leq n-1$. 

Next, we assume that $\gcd(n,i)=1$.
By Lemma \ref{Lem:conjun}, the irreducible $\hat{g}^i$-twisted $V_{L}$-module $V_{\lambda+L}[\hat{g}^i]$ is conjugate to $V_{L}[\hat{g}^i]$ by some automorphism of $V_{L}^{\hat{g}}$.
Hence, the assertion also holds by the case $\lambda\in N^*$ above.

Let us discuss the remaining case, namely, $\lambda \in L^*\setminus N^*$ and $\gcd(n,i)\neq 1$. 
By \eqref{Eq:L^*/N^*}, there exist $1 \leq k  \leq n-1$ and $\lambda'\in N^*$ such that $\lambda= k \chi_\Delta +\lambda'$.
By Remark \ref{R:intwist}, the vector space
$V_{k \chi_\Delta +\lambda' +N}[\hat{g}^i]$ forms an irreducible $\varphi^{-k} \hat{g}^i$-twisted  $V_N$-module. 
Take $s\in\Z$ and $0\le r<i$ such that $-k= si -r$.
Since $\hat{g}$ and $\varphi$ are commutative, we have $\varphi^{-k} \hat{g}^i= \varphi^{-r} (\varphi^{s} \hat{g})^i$. 

Now let $h_s\in \Aut(V_N)$ be defined as in Proposition \ref{P:extra}.  
Then $h_s$ and $\varphi$ are commutative, and hence by Proposition \ref{P:extra} (1), we have $h_s^{-1} (\varphi^{-k} \hat{g}^i) h_s =h_s^{-1}(\varphi^{-r}(\varphi^s\hat{g})^i)h_s=\hat{g}^r \varphi^i$.  
Thus, $V_{k \chi_\Delta +\lambda'+N}[\hat{g}^i]\circ h_s$  is an irreducible $\hat{g}^r \varphi^i$-twisted $V_N$-module.
Recall from $\varphi=id$ on $V_{L}$ that $V_{\lambda+L}[\hat{g}^i]$ is an irreducible $\hat{g}^i$-twisted $V_{L}$-submodule of $V_{k \chi_\Delta +\lambda' +N}[\hat{g}^i]$.
Hence $V_{\lambda+L}[\hat{g}^i]\circ h_s$ is an irreducible $\hat{g}^r$-twisted $V_{L}$-module.
Thus $V_{\lambda+L}[\hat{g}^i](q)\circ h_s\cong V_{\lambda'+L}[\hat{g}^r](q')$ as $V_{L}^{\hat{g}}$-modules for some $\lambda'\in L^*$ and $0\le q'\le n-1$.
Note that $r$ is strictly less than $i$. By repeating this process, we obtain the result.
\end{proof}

\section{Coinvarinat lattices of isometries of the Leech lattice}\label{S:coinv}
In this section, we describe five coinvariant lattices of isometries of the Leech lattice and summarize their properties.
All calculations on lattices in this section have been done by MAGMA (\cite{MAGMA}).

Let $\Lambda$ be the Leech lattice and let $g\in O(\Lambda)$ whose conjugacy class is $4C$, $6E$, $6G$, $8E$ or $10F$ (in \cite{ATLAS} notation).
Let $\Lambda_g$ be the coinvariant lattice of $g$ as in \eqref{Eq:coinv}.
Then $g$ acts fixed-point freely on $\Lambda_g$.
The following can be verified by MAGMA.

\begin{lemma}\label{Lem:conjclass0} 
Let $g\in O(\Lambda)$ whose conjugacy class is $4C$, $6E$, $6G$, $8E$ or $10F$.
\begin{enumerate}[{\rm (1)}]
\item $(1-g)\Lambda_g^*=\Lambda_g$.

\item The quotient group $C_{O(\Lambda_g)}(g)/\langle g\rangle$ acts faithfully on $\mathcal{D}(\Lambda_g)$. 
\end{enumerate}
\end{lemma}

The following can be found in \cite{HL90}.

\begin{lemma}\label{Lem:conjLeech}
For $g$ in the conjugacy class $4C$, $6E$, $6G$, $8E$ or $10F$ of $O(\Lambda)$, its conjugacy class is uniquely determined by $|g|$, $\rank \Lambda^g$ and $\mathcal{D}(\Lambda^g)(\cong\mathcal{D}(\Lambda_g))$.  
\end{lemma}

We now give an alternative construction of $\Lambda_g$.
Let $R$ be the root lattice and let $c\in R^*/R$ be the vector as in the corresponding row of Table \ref{glue}.
Let $C$ be the subgroup of $R^*/R$ generated by $c$.
Let $L_B(C)$ be the lattice given in \eqref{Eq:LBC}.
Then $L_B(C)$ is even.
In addition, one can verify that $\Lambda_g\cong L_B(C)$ by MAGMA.
Then $g$ induces an isometry of $L_B(C)$; we use the same symbol $g$.
On the other hand, $O(L_B(C))$ contains $g_{\Delta,c}$ defined as in Section \ref{S:extra}.
Then Assumptions (i) and (ii) in Section 4 hold and hence $g_{\Delta,c}$ is a fixed-point free isometry of order $n$ and $(1-g_{\Delta,c})L_B(C)^*\subset L_B(C)$.
Thus $g_{\Delta,c}$ lifts to an isometry of the Leech lattice $\Lambda$ which acts trivially 
on $\Lambda^g$.   
Here we view $\Lambda$ as an overlattice of $\Lambda^g\oplus L_B(C)$.  
Clearly, its fixed-point sublattice is also $\Lambda^g$ and by Lemma \ref{Lem:conjLeech}, $g$ is conjugate to $g_{\Delta,c}$ in $O(\Lambda)$.
Thus, we may assume that $g=g_{\Delta,c}$ under 
the identification  
$\Lambda_g\cong L_B(C)$.

\begin{proposition}\label{P:Nc} For $g$ in the conjugacy class $4C,6E,6G,8E$ or $10F$ of $O(\Lambda)$, the coinvariant lattice $\Lambda_g$ can be constructed as $L_B(C)$.
\end{proposition}

We summarize some properties of $\Lambda_g$ in Table \ref{glue}; the structures of $O(\Lambda_g)$ and $C_{O(\Lambda_g)}(g)$ are computed by MAGMA.
Here for $\Lambda_g^*/\Lambda_g$, the symbol $\prod{a_i}^{b_i}$ means the abelian group structure $\bigoplus(\Z/a_i\Z)^{b_i}$.
For the notations of groups, see \cite{ATLAS}
\begin{small}
\begin{longtable}{|c|c|c|c|c|c|}
\caption{Coinvariant lattices $\Lambda_g$}\label{glue}
\\ \hline 
Class&
$R$ & $c$&$\Lambda_g^*/\Lambda_g$&$O(\Lambda_g)$&$C_{O(\Lambda_g)}(g)$ \\ \hline
$4C$&
 $A_3^4A_1^2$ & $(1 1 1 1|11)$& $2^24^4$&
$2^{10+3}.\Sym_6$&$2^{9+3}.\Sym_6$\\
$6E$&
 $A_5^2 A_2^2 A_1^2$& $(11|11|11)$&
$2^43^4$&$6.(\GO^+_4(2)\times \GO^+_4(3)).2$&$6.(\GO^+_4(2)\times \GO^+_4(3))$\\
$ 6G$&
$A_5^3 A_1^3$& $(111|1 1 1) $&$2^63^3$&
$6.(\GO_3(3)\times\PSO^+_4(3)).2$&$6.(\GO_3(3)\times\PSO^+_4(3))$\\
$8E$&
$A_7^2A_3A_1$ & $(13|1|1)$&$2.4.8^2$&$2^6.(\Dih_8\times \Sym_4)$&$2^5.(4\times\Sym_4)$\\
$10F$&
 $A_9^2A_1^2$& $(13|11)$&$2^45^2$ &$(2\times\AGL_1(5)).\Dih_8^2$&$10.\Dih_8^2$\\
\hline 
\end{longtable}
\end{small}

Next, we discuss the quadratic space $(\mathcal{D}(\Lambda_g),q)$ and the orthogonal group $O(\mathcal{D}(\Lambda_g))$ (see Section \ref{S:lattice}).
For $g\in O(\Lambda)$ and $k\in \Z_{>0}$,  set
\begin{equation}
\mathcal{L}_{g,k}=\{ \lambda+\Lambda_g\in \mathcal{D}(\Lambda_g)\mid q(\lambda+\Lambda_g)=0,\ o(\lambda+\Lambda_g)=k\}, \label{Def:Lg2}
\end{equation}
where $o(\lambda+\Lambda_g)$ is the order of $\lambda+\Lambda_g$ in $\mathcal{D}(\Lambda_g)$.
By using MAGMA, we can also verify the following two lemmas.

\begin{lemma}\label{Lem:conjclass1} Let $g$ be an isometry of $\Lambda$ whose conjugacy class is $4C$, $6E$, or $8E$.
Let $k$ be a divisor of $|g|$.
Then $C_{O(\Lambda_g)}(g)/\langle g\rangle$ acts transitively on $\mathcal{L}_{g,k}$. 
\end{lemma}

\begin{lemma}\label{Lem:conjclass6G10F}
Let $g$ be an isometry of $\Lambda$ whose conjugacy class is $6G$ or $10F$.
Then, the group $C_{O(\Lambda_g)}(g)/\langle g\rangle$ acts transitively on $\mathcal{L}_{g, \, |g|/2}$.
\end{lemma}

We summarize some properties of $O(\mathcal{D}(\Lambda_g))$ in Table \ref{quad}.

\begin{longtable}{|c|c|c|c|c|}
\caption{Orthogonal groups $O(\mathcal{D}(\Lambda_g))$}\label{quad}
\\ \hline 
Class& $\mathcal{D}(\Lambda_g)$&$O(\mathcal{D}(\Lambda_g))$
&$C_{O(\Lambda_g)}(g)/\langle g\rangle$
& $|O(\mathcal{D}(\Lambda_g)):C_{O(\Lambda_g)}(g)/\langle g\rangle|$ \\ \hline
$4C$& $2^24^4$&$2^{11}.\Sym_6$
&$2^{10}.\Sym_6$&$2$ \\
$6E$&$2^43^4$&$\GO_4^+(2)\times \GO_4^+(3)$
&$\GO_4^+(2)\times \GO_4^+(3)$ &$1$\\ 
$ 6G$&$2^63^3$&$\GO_3(3)\times \PSO^+_4(3).2^2$ & $\GO_3(3)\times\PSO^+_4(3)$&$4$\\
$8E$& $2.4.8^2$&$2^6.\Dih_{12}$&$2^6.\Dih_6$&$2$ \\
$10F$&$2^45^2$ &$\Dih_8\times \Sym_4$&$\Dih_8^2$&3\\
\hline 
\end{longtable}

\section{Automorphism group of the cyclic orbifold VOA $V_{\Lambda_g}^{\hat{g}}$}\label{S:6}
Let $g$ be an isometry of the Leech lattice $\Lambda$ whose conjugacy class is $4C,6E,6G,8E$ or $10F$ and let $\Lambda_g$ be the coinvariant lattice of $g$.
Let $n$ be the order of $g$.
Then $g$ acts on $\Lambda_g$ as a fixed-point free isometry of order $n$.
Let $\hat{g}\in\Aut(V_{\Lambda_g})$ be a standard lift of $g$.
Note that the order of $\hat{g}$ is also $n$.
In this section, we determine the automorphism group of the cyclic orbifold VOA $V_{\Lambda_g}^{\hat{g}}$.

It follows from \eqref{Eq:hom2} and Lemma \ref{Lem:conjclass0} (1) that 
$$\hom(\Lambda_g/(1-g)\Lambda_g, \Z/n\Z)\cong ((1-g)^{-1}\Lambda_g^*)/\Lambda_g^*\cong \Lambda_g^*/(1-g)\Lambda_g^*=\Lambda_g^*/\Lambda_g.$$
Hence by Theorem \ref{normalizer},
\begin{equation}
C_{\aut(V_{\Lambda_g})}(\hat{g})/ \langle \hat{g}\rangle \cong 
(\Lambda_g^*/\Lambda_g).(C_{O(\Lambda_g)}(g)/\langle g\rangle) \label{Eq:exactC}
\end{equation} is a subgroup of $\aut(V_{\Lambda_g}^{\hat{g}})$. 

Let $\mathrm{Irr}(V_{\Lambda_g}^{\hat{g}})$ be the set of (isomorphism classes of) irreducible $V_{\Lambda_g}^{\hat{g}}$-modules.
Then $\aut(V_{\Lambda_g}^{\hat{g}})$ acts on $\mathrm{Irr}(V_{\Lambda_g}^{\hat{g}})$ by conjugation as in \eqref{Eq:conjact}.
Recall from Theorem \ref{Thm:EMS} that the map $q$ in \eqref{Eq:qVOA} gives a quadratic space structure on $\irr(V_{\Lambda_g}^{\hat{g}})$.
By Lemma \ref{Lem:conj} (2), there exists a natural group homomorphism $$\mu: \aut(V_{\Lambda_g}^{\hat{g}}) \to O(\mathrm{Irr}(V_{\Lambda_g}^{\hat{g}})),$$ where $O(\mathrm{Irr}(V_{\Lambda_g}^{\hat{g}}))$ is the orthogonal group of the quadratic space $(\mathrm{Irr}(V_{\Lambda_g}^{\hat{g}}), q)$.

\begin{proposition}\label{conj_mu}
The group homomorphism $\mu$ is injective.
\end{proposition}

\begin{proof}
 Let $\tau\in \ker \mu$. Then $\tau$ fixes $\{V_{\Lambda_g}(j) \mid j=0, 1, \dots, 
|g|-1\}$ pointwisely.
By Proposition \ref{P:stabVL1}, $\tau\in C_{\aut(V_{\Lambda_g})}(\hat{g})/\langle \hat{g} \rangle$.
By Theorem \ref{normalizer}, $\bar{\tau}\in C_{O(\Lambda)}(g)/\langle g\rangle$. 
By Lemmas \ref{Lem:conjlift} and \ref{Lem:conjclass0}, we have $\bar{\tau}=1$.
By Theorem \ref{normalizer} again, $\tau\in \hom(\Lambda_g/(1-g)\Lambda_g,\, \Z/n\Z)$.
By Lemmas \ref{Re:conj} and \ref{Lem:conjclass0} (1), $\tau$ must be the identity map.
\end{proof}

\begin{remark} By the proposition above, we can view $\aut(V_{\Lambda_g}^{\hat{g}})$ as a subgroup of $O(\mathrm{Irr}(V_{\Lambda_g}^{\hat{g}}))$.
\end{remark}

Next we will prove that $\Aut(V_{\Lambda_g}^{\hat{g}})$ acts transitively on
a certain subset $\mathcal{S}_g$  of $\irr(V_{\Lambda_g}^{\hat{g}})$ (cf. \eqref{Eq:Sg} and \eqref{Eq:Sg2} below). 
In Sections \ref{S:non-db} and \ref{S:db} below, we will discuss the case where $|\hat{g}|=|g|$ and $|\hat{g}|=2|g|$ on $V_\Lambda$, respectively.

\subsection{Transitivity for the cases $4C,6E$ and $8E$ }\label{S:non-db}
In this subsection, we assume that $g\in O(\Lambda)$ belongs to the conjugacy classes $4C,6E$ or $8E$.
Then $|\hat{g}|=|g|$ on both $V_\Lambda$ and $V_{\Lambda_g}$; we set $n=|g|$.

We now review the quadratic space structure of $(\irr(V_{\Lambda_g}^{\hat{g}}),q)$ from \cite[Sections 6.1, 6.2 and 6.3]{La19} (see also \cite[Section 5.2.1]{La19} for more general cases).
We choose a labelling of the irreducible $V_{\Lambda_g}^{\hat{g}}$-module $V_{\lambda+\Lambda_g}[\hat{g}^i](j)$ for $\lambda+\Lambda_g\in\mathcal{D}(\Lambda_g)$ and $0\le i,j\le n-1$ as in Section \ref{S:irred} (cf.\ \cite[Theorem 5.3]{La19}).
Then for $0\le i,j\le n-1$,
\begin{equation}
q(V_{\lambda+\Lambda_g}[\hat{g}^i](j))
\equiv \varepsilon(V_{\lambda+\Lambda_g}[\hat{g}^i](j))\equiv \frac{ij}{n}+\frac{(\lambda|\lambda)}{2}\mod \Z.\label{Eq:quad}
\end{equation}
Note that $\varepsilon(\irr(V_{\Lambda_g}^{\hat{g}}))\subset (1/n)\Z$.
In addition, by \cite[Theorem 5.3]{La19}, as groups, $\irr(V_{\Lambda_g}^{\hat{g}})\cong \mathcal{D}(\Lambda_g)\times(\Z/n\Z)^2$; the explicit multiplications are given as follows
\begin{equation}
V_{\lambda_1+\Lambda_g}[\hat{g}^{i_1}](j_1)\boxtimes 
V_{\lambda_2+\Lambda_g}[\hat{g}^{i_2}](j_2)=
V_{\lambda_1+\lambda_2+\Lambda_g}[\hat{g}^{i_1+i_2}](j_1+j_2)\label{Eq:fusion2}.
\end{equation}
Note that the two subgroups $$\{V_{\lambda+\Lambda_g}(j)\mid 0\le j\le |g|-1\}\cong\mathcal{D}(\Lambda_g)\quad \text{and}\quad \{V_{\Lambda_g}[\hat{g}^i](j)\mid 0\le i,j\le |g|-1\}\cong(\Z/n\Z)^2$$
are orthogonal with respect to the bilinear form.
Hence as quadratic spaces
\begin{equation}
(\irr(V_{\Lambda_g}^{\hat{g}}),q)\cong (\mathcal{D}(\Lambda_g),q)\oplus((\Z/n\Z)^2,q),\label{Eq:quado}
\end{equation}
where $q$ on $\mathcal{D}(\Lambda_g)$ is defined as in \eqref{Eq:qL} and $q((i,j))\equiv ij/n\pmod\Z$ on $(\Z/n\Z)^2$.
For $i\in\Z/n\Z$, we denote by $o(i)$ the order of $i$ in $\Z/n\Z$.

Now consider the subset 
\begin{equation}
\mathcal{S}_g= \{ M\in  \mathrm{Irr}(V_{\Lambda_g}^{\hat{g}})\mid q(M)=0,\ o(M)=n\},\label{Eq:Sg}
\end{equation} 
where $o(M)$ denotes the order of $M$ in the abelian group $\mathrm{Irr}(V_{\Lambda_g}^{\hat{g}})$.
Note that $V_{\Lambda_g}(1)\in\mathcal{S}_g$.
By the group structure and the conformal weights of irreducible $V_{\Lambda_g}^{\hat{g}}$-modules in \eqref{Eq:quad}, \eqref{Eq:fusion2} and \eqref{Eq:quado}, we obtain the following:

\begin{lemma}\label{L:untwstS} Let $\lambda+\Lambda_g\in \mathcal{D}(\Lambda_g)$ and $0\le j\le n-1$.
Then, the irreducible $V_{\Lambda_g}^{\hat{g}}$-module $V_{\lambda+\Lambda_g}(j)$ belongs to $\mathcal{S}_g$ if and only if both $\lcm(o(\lambda+\Lambda_g),o(j))=n$ and $q(\lambda+\Lambda_g)=0$ hold.
\end{lemma}

We now discuss the conjugates of elements in $\mathcal{S}_g$ under $\Aut(V_{\Lambda_g}^{\hat{g}})$.
By Proposition \ref{P:Nc}, $\Lambda_g\cong L_B(C)$ and let $N\cong L_A(C)$ be the overlattice of $\Lambda_g$ as in Section \ref{S:coinv}.
Then $N/\Lambda_g\cong\Z/n\Z$;
let $\gamma\in N$ such that $N/\Lambda_g=\langle\gamma+\Lambda_g\rangle$ and $(\chi_\Delta|\gamma)\in (1/n)+\Z$.
Then $o(\gamma+\Lambda_g)=n$ and $q(\gamma+\Lambda_g)=0$.

\begin{lemma}\label{L:gammaij} Any element in $\{V_{s\gamma+\Lambda_g}(j)\mid \lcm(o(s),o(j))=n\}$ is conjugate to an element in $\{V_{s\gamma+\Lambda_g}(j)\mid o(s)=n,\ 0\le j\le n-1\}$ by an element in $\Aut(V_{\Lambda_g}^{\hat{g}})$.
\end{lemma}
\begin{proof} Consider $V_{s\gamma+\Lambda_g}(j)$ with $\lcm(o(s),o(j))=n$ and $o(s)<n$.
By the assumption on $g$, we have $n=4,6$ or $8$.
If $n=o(j)$, then we have done by substituting $0$ for $k$ in Proposition \ref{P:extra} (2), which proves the cases where $n=4$ and $n=8$.

Assume that $n=6$ and $o(j)<n$.
Then $(o(s),o(j))=(2,3)$ or $(3,2)$.
Substituting $1$ for $k$ in Proposition \ref{P:extra}, we see that $V_{s\gamma+\Lambda_g}(j)$ is conjugate to $V_{j\gamma+\Lambda_g}(\ell)$ with $o(\ell)=6$.
By the argument above, we obtain the result.
\end{proof}

By the definition of $\mathcal{S}_g$ in \eqref{Eq:Sg}, $\Aut(V_{\Lambda_g}^{\hat{g}})$ preserves $\mathcal{S}_g$.
In addition, we obtain the following, which is crucial for our arguments below.

\begin{proposition}\label{P:tranSg} $\Aut(V_{\Lambda_g}^{\hat{g}})$ acts transitively on $\mathcal{S}_g$.
\end{proposition}

\begin{proof} 
By Lemma \ref{Lem:conjclass0}, $(1-g)\Lambda_g^*=\Lambda_g$.
By Proposition \ref{untwisted} and Lemma \ref{L:untwstS}, under $\Aut(V_{\Lambda_g}^{\hat{g}})$, any element in $\mathcal{S}_g$ is conjugate to $V_{\lambda+\Lambda_g}(j)$  for some $\lambda+\Lambda_g\in \mathcal{D}(\Lambda_g)$ and $0\le j\le n-1$ satisfying $\lcm(o(\lambda+\Lambda_g),o(j))=n$.
It follows from $q(V_{\lambda+\Lambda_g}(j))=0$ that $q(\lambda+\Lambda_g)=0$ by \eqref{Eq:quad}.
Set $k=o(\lambda+\Lambda_g)$.
We may assume that $k\neq 0$; otherwise, by Theorem \ref{T:extra} and Proposition \ref{P:Nc}, there exists $\xi\in\Aut(V_{\Lambda_g}^{\hat{g}})$ such that $V_{\Lambda_g}(j)\circ\xi\cong V_{\mu+\Lambda_g}(j')$ for some $\mu\notin\Lambda_g$.
Then $\lambda+\Lambda_g\in\mathcal{L}_{n,k}$, where $\mathcal{L}_{n,k}$ is defined as in \eqref{Def:Lg2}.
Set $k'=n/k$.
Then $q(k'\gamma+\Lambda_g)=0$ and $o(k'\gamma+\Lambda_g)=k$, and hence $k'\gamma+\Lambda_g\in \mathcal{L}_{g,k}$.
By Proposition \ref{normalizer} and Lemmas \ref{Lem:conjlift} and \ref{Lem:conjclass1}, $V_{\lambda+\Lambda_g}(j)$ is conjugate to $V_{k'\gamma+\Lambda_g}(j')$ for some $j'$.
Since $o(V_{k'\gamma+\Lambda_g}(j'))$ is still $n$, Lemma \ref{L:gammaij} shows $\lcm(o(k'\gamma+\Lambda_g),o(j'))=n$.
By Lemmas \ref{Lem:conjclass1} and \ref{L:gammaij}, $V_{k'\gamma+\Lambda_g}(j')$ is conjugate to $V_{\gamma+\Lambda_g}(j'')$ for some $j''$.
Recall that $(\chi_\Delta|\gamma)\in(1/n)+\Z$.
Then by Lemma \ref{Lem:conjhom}, $V_{\gamma+\Lambda_g}(i'')$ is conjugate to $V_{\gamma+\Lambda_g}(0)$.
Therefore we have proved this proposition.
\end{proof}

\subsection{Transitivity for the cases $6G$ and $10F$.}\label{S:db}
In this subsection, we assume that $g\in O(\Lambda)$ belongs to conjugacy classes $6G$ or $10F$.
Set $n=|g|$ and $p=n/2$.
Then $|\hat{g}|=2|g|=2n$ on $V_\Lambda$, and $|\hat{g}|=|g|=n$ on $V_{\Lambda_g}$.

We now recall the quadratic space structure of $\irr(V_{\Lambda_g}^{\hat{g}})$ from \cite{La19}. 
In order to do it, we have to describe certain vectors $u$ and $h$ in $\Lambda_g^*\cap(1/2)\Lambda_g$ as in \cite[Section 5.2.2]{La19}; in the following, we recall how to choose $u$ and $h$.
By \cite{ATLAS}, $g^p\in O(\Lambda)$ belongs to the conjugacy class $2C$; thus, $\Lambda_{g^p} \cong \sqrt{2}D_{12}^+$ and $2\Lambda_{g^p}^*= \Lambda_{g^p}$. Note that $2\Lambda_{g^p}^*\cong \sqrt{2}D_{12}^+$ is not doubly even since $\sqrt2D_{12}^+$ has a norm $6$ vector.
As in \cite[Section 5.2.2]{La19}, let  
$
E= \{ x\in (1/2) \Lambda_{g^p}\mid (x|x)\in \Z\}. 
$
Then $2E$ is a maximal doubly even sublattice of $\Lambda_{g^p} \cong \sqrt{2}D_{12}^+$ and thus $2E\cong\sqrt{2}D_{12}$. 
Let $h$ be a minimal norm vector of $E^*\cong\sqrt2D_{12}^*$.
Then $(h|h)=2$ and $(h|x)\in \Z$ for all $x\in E$ and $(h|x)\in 1/2+\Z$ for all $x\in (1/2)\Lambda_{g^p}\setminus E$.
Let $u$ be a minimal norm vector in $(1/2)\Lambda_{g^p}\setminus E$.
Then $(u|u)=3/2$.
We may choose $u$ so that $(h|u)=1/2$ by using explicit description of $D_{12}$.
Then $(1/2)\Lambda_{g^p}= E \cup (u+E)$ and $E^*=((1/2)\Lambda_{g^p})^*\cup(h+((1/2)\Lambda_{g^p})^*)$. 
We also may assume $u,h\in\Lambda_g^*\cap(1/2)\Lambda_g$

Let us describe $h$ and $u$ under the identification $\Lambda_g\cong L_B(C)$ as in  Section \ref{S:coinv}.
Set $L=L_B(C)$ and $N=L_A(C)$. 
Recall that $L^*= \bigcup_{i=0}^{n-1}(i\chi_\Delta +N^*)$ and $N=\bigcup_{i=0}^{n-1}(i\gamma+L)$.
Let $h$ be a simple root of the root sublattice of $N$ of type $A_1$ (cf.\ Table \ref{glue}) 
and let $u$ be a minimal norm vector in $(n/2)\chi_\Delta+L$.
Then $h,u\in L^*\cap(1/2)L$ and $(h|h)=2$, $(u|u)=3/2$ and $(u|h)=\pm1/2$; we may assume $(u|h)=1/2$ by replacing $h$ by $-h$ if necessary.
In addition, $h$ and $u$ belong to the $-1$-eigenspace of $g^p$.
Hence, viewing $h,u$ as vector in $\Lambda$ under the isomorphism $\Lambda_g^*\cong L^*$, we have $h,u\in\Lambda_{g^p}^*\cap\Lambda_g^*$.
Note that 
\begin{equation}
h\in (n/2)\gamma+\Lambda_g.\label{Eq:h}
\end{equation}
Thus, these $h$ and $u$ agree with the choice discussed above.

Set $$Y_g=\{v\in \Lambda_g^*\mid (v|u),(v|h)\in\Z\}.$$ 
Then $\Lambda_g^*/Y_g$ is isomorphic to $(\Z/2\Z)^2$ since it is generated by $ u+Y_g$ and $h+Y_g$.
It follows from \cite[Theorem 5.8]{La19} that as groups \begin{align*}
\irr(V_{\Lambda_g}^{\hat{g}})=\{V_{\lambda+pu+qh+\Lambda_g}[{\hat{g}^i}](j)\mid \lambda+\Lambda_g\in Y_g/\Lambda_g,\ 0\le i,j\le n-1,\ p,q\in\{0,1\}\}
\end{align*}
Here we choose labelling of $V_{\lambda+pu+qh+\Lambda_g}[{\hat{g}^i}](j)$ as in \cite[Section 5.2.2]{La19}.
In addition, as groups under the fusion products, 
\begin{align}
\irr(V_{\Lambda_g}^{\hat{g}})=\{V_{\lambda+\Lambda_g}(0)\mid \lambda+\Lambda_g\in Y_g/\Lambda_g\}\times\{V_{pu+qh+\Lambda_g}[{\hat{g}^i}](j)\mid 0\le i,j\le n-1,\ p,q\in\{0,1\}\}\label{Eq:irrg22} 
\end{align}
Note that as groups
\begin{equation}
\{V_{\lambda+\Lambda_g}(0)\mid \lambda+\Lambda_g\in  Y_g/\Lambda_g\}\cong Y_g/\Lambda_g,\label{Eq:irrg20}
\end{equation}
and 
\begin{equation}
\{V_{pu+qh+\Lambda_g}[{\hat{g}^i}](j)\mid 0\le i,j\le n-1,\  p,q\in\{0,1\}\}\cong (\Z/2n\Z)^2.\label{Eq:irrg21}
\end{equation}
The isomorphism 
\begin{equation}
\irr(V_{\Lambda_g}^{\hat{g}})\cong Y_g/\Lambda_g\times(\Z/2n\Z)^2\label{Eq:gp2n}
\end{equation} is explicitly given in \cite[(5.15)]{La19} as follows:
\begin{equation}
(\lambda+\Lambda_g,(i,j))\mapsto \begin{cases}V_{\lambda+u+\Lambda_g}[\hat{g}^i]\left(\lfloor\frac{j}{2}\rfloor\right)&(i+j\in2\Z+1,\ 0\le i\le n-1),\\
V_{\lambda+\Lambda_g}[\hat{g}^i]\left(\lfloor\frac{j}{2}\rfloor\right)&(i+j\in2\Z,\ 0\le i\le n-1),\\
V_{\lambda+h+u+\Lambda_g}[\hat{g}^i]\left(\lfloor\frac{j}{2}\rfloor\right)&(i+j\in2\Z+1,\ n\le i\le 2n-1),\\
V_{\lambda+h+\Lambda_g}[\hat{g}^i]\left(\lfloor\frac{j}{2}\rfloor\right)&(i+j\in2\Z,\ n\le i\le 2n-1).
\end{cases}\label{Eq:ij}
\end{equation}
For example, $(\Lambda_g,(1,0))$ and $(\Lambda_g,(0,1))$ correspond to $V_{u+\Lambda_g}[\hat{g}](0)$ and $V_{u+\Lambda_g}(0)$, respectively.
We often use these notations.
In addition, by \cite[Theorem 5.8]{La19}, the quadratic form $q:\irr(V_{\Lambda_g}^{\hat{g}})\to \Q/\Z$ is given by
\begin{equation}
q((\lambda+\Lambda_g,(i,j)))\equiv \varepsilon((\lambda+\Lambda_g,(i,j)))\equiv \frac{(\lambda|\lambda)}{2}+\frac{ij}{2n}+\frac{3}{4}(\overline{i+j})\pmod\Z,\label{Eq:6Gq}
\end{equation}
where $\overline{i+j}=0$ and $1$ if $i+j\in2\Z$ and $i+j\in2\Z+1$, respectively.
Note that $\varepsilon(\irr(V_{\Lambda_g}^{\hat{g}}))\subset (1/2n)\Z$.
Hence the two subgroups \eqref{Eq:irrg20} and \eqref{Eq:irrg21} are orthogonal with respect to the associated bilinear form.

By Table \ref{glue}, the exponent of $Y_g/\Lambda_g$ is $n$. 
By \eqref{Eq:gp2n}, the prime factors of the order of $\irr(V_{\Lambda_g}^{\hat{g}})$ are $2$ and $p=n/2$.
Let $O_r(A)$ denote the maximal $r$-subgroup of the abelian group $A$ for a prime $r$.
It follows from $\gcd(p,2)=1$ that as quadratic space, 
\begin{equation}\label{Eq:6GO23}
(\irr(V_{\Lambda_g}^{\hat{g}}),q)\cong (O_2(\irr(V_{\Lambda_g}^{\hat{g}})),q)\times (O_{p}(\irr(V_{\Lambda_g}^{\hat{g}})),q);
\end{equation}
note that the subgroups $O_2(\irr(V_{\Lambda_g}^{\hat{g}}))$ and $(O_{p}(\irr(V_{\Lambda_g}^{\hat{g}}))$ are commutative.
Obviously, any element of $O(\irr(V_{\Lambda_g}^{\hat{g}}))$ preserves both $O_2(\irr(V_{\Lambda_g}^{\hat{g}}))$ and $O_{p}(\irr(V_{\Lambda_g}^{\hat{g}}))$.
Hence as groups
\begin{equation}\label{Eq:6G023g}O(\irr(V_{\Lambda_g}^{\hat{g}}))\cong O(O_2(\irr(V_{\Lambda_g}^{\hat{g}})))\times O(O_{p}(\irr(V_{\Lambda_g}^{\hat{g}}))).
\end{equation}

Now, we consider the set 
\begin{equation}
\mathcal{S}_g= \{ 2a  \mid a\in \mathrm{Irr}(V_{\Lambda_g}^{\hat{g}}),\ o(a)=2n,\ q(2a)=0\},\label{Eq:Sg2}
\end{equation} 
where $o(a)$ denotes the order of $a$ in the abelian group $\mathrm{Irr}(V_{\Lambda_g}^{\hat{g}})$.
Note that $V_{\Lambda_g}(1)\in\mathcal{S}_g$ since $V_{\Lambda_g}(1)$ corresponds to $(\Lambda_g,(0,2))=2(\Lambda_g,(0,1))$.
Set 
\begin{align*}
\mathcal{S}_g(2)&=\{2a\mid a\in O_2(\irr(V_{\Lambda_g}^{\hat{g}})),\ 
o(a)=4,\ q(2a)=0\},\\
\mathcal{S}_g(p)&=\{a\in O_p(\irr(V_{\Lambda_g}^{\hat{g}}))\mid o(a)=p,\ q(a)=0\}.
\end{align*}
By the quadratic space structure of $\irr(V_{\Lambda_g}^{\hat{g}})$ in \eqref{Eq:gp2n},  \eqref{Eq:ij} and \eqref{Eq:6Gq}, we obtain
\begin{align}
\mathcal{S}_g(2)&=\{(\Lambda_g,(n,0)),\ (\Lambda_g,(0,n)),\ (\Lambda_g,(n,n))\},\label{Eq:Sg22}\\
\mathcal{S}_g(p)&=\{(\mu+\Lambda_g,(4i,4j))\mid \mu+\Lambda_g\in O_p(Y_g/\Lambda_g),\ \frac{(\mu|\mu)}{2}+\frac{ij}{p}\in\Z\}\setminus\{(\Lambda_g,(0,0))\}.\label{Eq:Sgp}
\end{align}
Note that $O_p(\irr(V_{\Lambda_g}^{\hat{g}}))$ is elementary abelian.
By \eqref{Eq:6Gq}, for $a\in O_2(\irr(V_{\Lambda_g}^{\hat{g}}))$ (resp. $a\in O_p(\irr(V_{\Lambda_g}^{\hat{g}}))$), the value $q(a)$ belongs to $(1/4)\Z$ (resp. $(1/p)\Z$).
Hence, under the isomorphism \eqref{Eq:6GO23}, we obtain $$\mathcal{S}_g=\{(a,b)\mid a\in \mathcal{S}_g(2),\ b\in\mathcal{S}_g(p)\}.$$
We will prove that $\Aut(V_{\Lambda_g}^{\hat{g}})$ acts transitively on $\mathcal{S}_g$.

\begin{lemma}\label{L:transg2} $\aut(V_{\Lambda_g}^{\hat{g}})$ acts transitively on $\mathcal{S}_g(2)$.
\end{lemma}
\begin{proof} 
By \eqref{Eq:ij} and \eqref{Eq:Sg22}, 
\begin{equation}
\mathcal{S}_g(2)=\{V_{\Lambda_g}(p),V_{h+\Lambda_g}(0),V_{h+\Lambda_g}(p)\}.\label{Eq:set2}
\end{equation}
By Proposition \ref{P:extra}, along with $\Lambda_g\cong L_B(C)$ and the choice of $h$ in \eqref{Eq:h}, $V_{\Lambda_g}(p)$ and $V_{h+\Lambda_g}(0)$ are conjugate under $\Aut(V_{\Lambda_g}^{\hat{g}})$.
It follows from $(u|h)=1/2=p/n$ and Lemma \ref{Lem:conjhom} that $V_{h+\Lambda_g}(0)\cong V_{h+\Lambda_g}(p)\circ \sigma_{(1-g)^{-1}u}$.
\end{proof}
\begin{lemma}\label{6GO3}
$\aut(V_{\Lambda_g}^{\hat{g}})$ acts transitively on $\mathcal{S}_g(p)$.
\end{lemma}

\begin{proof}
Recall from Lemma \ref{Lem:conjclass0} that $(1-g)\Lambda_g^*=\Lambda_g$.
By Proposition \ref{untwisted}, any element of $\mathcal{S}_g(p)$ is conjugate to $V_{\mu+\Lambda_g}(\ell)$ for some $\mu+\Lambda_g\in\mathcal{D}(\Lambda_g)$ and $0\le \ell\le n-1$.
By the group structure of $\irr(V_{\Lambda_g}^{\hat{g}})$ in \eqref{Eq:gp2n} and \eqref{Eq:ij}, $V_{\mu+\Lambda_g}(\ell)\in\mathcal{S}_g(p)$ belongs to one of the following:

\begin{enumerate}[(I)]
\item $\{V_{\Lambda_g}(2j)\mid 1\le j\le p-1\}$;
\item 
$\{V_{\mu+\Lambda_g}(2j)\mid \mu+\Lambda_g\in \mathcal{L}_{g,p}
,\ 0\le j\le p-1\}$.
\end{enumerate}

By Lemmas \ref{Lem:conjlift}, \ref{Lem:conjhom} and \ref{Lem:conjclass6G10F}, any two elements in (II) are conjugate by some element in $C_{\Aut(V_{\Lambda_g})}(\hat{g})/\langle\hat{g}\rangle\subset\Aut(V_{\Lambda_g}^{\hat{g}})$.
By Proposition \ref{P:extra}, any element in (I) is also conjugate to some element in (II).
Therefore we obtain this lemma.
\end{proof}

\begin{lemma}\label{L:sigmau}
The automorphism $\sigma_{(1-g)^{-1}u}$ permutes two elements in $\mathcal{S}_g(2)$, and it fixes every element in $\mathcal{S}_g(p)$.
\end{lemma}
\begin{proof} The former assertion follows from $(u|h)=1/2$, Lemma \ref{Lem:conjhom} and \eqref{Eq:set2}.

It follows from $(u|Y_g)\subset\Z$ and Lemma \ref{Lem:conjhom} that $\sigma_{(1-g)^{-1}u}$ fixes $(\mu+\Lambda_g,(0,0))$ for any $\mu+\Lambda_g\in Y_g/\Lambda_g$. 
By \eqref{Eq:Sgp}, it suffices to prove that $\sigma_{(1-g)^{-1}u}$ fixes every element in $\{(4i,4j)\mid 0\le i,j< p\}$, the set of all order $1$ and $p$ elements in $(\Z/2n\Z)^2$.
By Lemma \ref{Lem:conjun}, $\sigma_{(1-g)^{-1}u}$ exchanges $(\Lambda_g,(1,1))=V_{\Lambda_g}[\hat{g}](0)$ and $(\Lambda_g,(1,0))=V_{u+\Lambda_g}[\hat{g}](0)$ (cf. \eqref{Eq:ij}).
Hence $\sigma_{(1-g)^{-1}u}$ fixes $(2,1)\in(\Z/2n\Z)^2$, and $4(2,1)=(8,4)\in(\Z/2n\Z)^2$.
It also fixes $(\Lambda_g,(4,0))=V_{\Lambda_g}(2)$ by Lemma \ref{Lem:conjhom}.
Hence $\sigma_{(1-g)^{-1}u}$ fixes every element in $\langle (8,4),(4,0)\rangle=\{(4i,4j)\mid  0\le i,j< p\}$.
\end{proof}

\begin{proposition}\label{6GS}
$\aut(V_{\Lambda_g}^{\hat{g}})$ acts transitively on the set 
$\mathcal{S}_g$.
\end{proposition}

\begin{proof}
Recall from \eqref{Eq:set2} that $|\mathcal{S}_g(2)|=3$; let $\mathcal{S}_g(2)=\{a,b,c\}$.
Set $G=\Aut(V_{\Lambda_g}^{\hat{g}})$.
Let $G_a$ be the stabilizer of $a$ in $G$.
By Lemma \ref{L:transg2}, we have $|G:G_a|=3$.
Take $y\in \mathcal{S}_g(p)$ and set $\mathcal{O}=\{gy\mid g\in G_a\}$.
If $\mathcal{O}=\mathcal{S}_g(p)$, then we obtain the result by Lemma \ref{L:transg2}.

Suppose, for a contradiction, that $\mathcal{O}\neq\mathcal{S}_g(p)$.
Then by Lemmas \ref{L:transg2} and \ref{6GO3}, $|\mathcal{S}_g(p)|/|\mathcal{O}|=3$; there exist $G_a$-orbits $\mathcal{O}'$ and $\mathcal{O}''$ such that $\mathcal{S}_g(p)=\mathcal{O}\cup\mathcal{O}'\cup\mathcal{O}''$ and the $G$-orbit of $(a,y)$ must be \begin{equation}\label{Eq:orb}
(\{a\}\times \mathcal{O})\cup(\{b\}\times\mathcal{O}')\cup(\{c\}\times\mathcal{O}'').
\end{equation}
However, this is not a $G$-orbit by Lemma \ref{L:sigmau}, which is a contradiction.
\end{proof}

\subsection{Shapes of the automorphism groups of $V_{\Lambda_g}^{\hat{g}}$.}

In this subsection, we describe the shapes of the automorphism groups of $V_{\Lambda_g}^{\hat{g}}$ for $g\in 4C,6E,6G,8E$, and $10F$.

The quadratic space structure of $(\irr(V_{\Lambda_g}^{\hat{g}}),q))$  and the subset $\mathcal{S}_g$ of $\irr(V_{\Lambda_g}^{\hat{g}})$ are given in \eqref{Eq:quad} and \eqref{Eq:Sg} (resp in \eqref{Eq:quado} and \eqref{Eq:Sg2}) if $g$ is in the conjugacy class $4C,6G$ or $8E$ (resp. $6E$ or $10F$).
By using MAGMA, we obtain the following proposition; note that in Table \ref{T:quadVOA}, $[2^s]$ means a $2$-group of order $2^s$.

\begin{proposition}\label{P:orthg}
Let $g\in O(\Lambda)$ whose conjugacy class is $4C,6E,6G,8E$ or $10F$.
Then the shape of the orthogonal group $O(\irr(V_{\Lambda_g}^{\hat{g}}))$ is given as in Table \ref{T:quadVOA}.
Moreover, 
$O(\irr(V_{\Lambda_g}^{\hat{g}}))$ acts transitively on $\mathcal{S}_g$
and the shape of the stabilizer of an element in $\mathcal{S}_g$
is also given as in Table \ref{T:quadVOA}.
\end{proposition}
\begin{longtable}{|c|c|c|}
\caption{Orthogonal groups $O(\irr(V_{\Lambda_g}^{\hat{g}}))$}\label{T:quadVOA}
\\ \hline 
Class& $\irr(V_{\Lambda_g}^{\hat{g}})$&$O(\irr(V_{\Lambda_g}^{\hat{g}}))$ \\ \hline
$4C$& $2^24^6$&$2^{22}.\GO_7(2)$ \\
$6E$&$2^63^6$&$\GO_6^+(2)\times \GO_6^+(3)$\\ 
$6G$&$2^44^23^5$&$2^{5+4}.(\GO_4^+(2)\times \Sym_3)\times \GO_5(3)$\\
$8E$& $2.4.8^4$&$2^{12+9}.\Sym_6$ \\
$10F$&$2^24^25^4$&$(Q_8:2^2).(\Sym_3\times\Sym_3)\times \GO_4^+(5)$\\
\hline \hline
Class& Stabilizer &$C_{\Aut(V_{\Lambda_g})}(\hat{g})/\langle \hat{g}\rangle$\\ \hline
$4C$&$2^{16+5}.\Sym_6$&$[2^{20}].\Sym_6$\\
$6E$&$6^4.(\GO_4^+(2)\times \GO_4^+(3))$&$6^4.(\GO_4^+(2)\times \GO_4^+(3))$\\
$6G$&$2^{5+4}.(GO_4^+(2)\times2)\times 3^3.(2\times\Sym_4)$&$2^3.6^3.(2\times \Sym_4\times\Alt_4^2).2$\\
$8E$&$(2^6\times 4^2).2^6.\Sym_3$&$[2^{15}].\Sym_3$\\
$10F$&$(Q_8:2^2).(\Sym_3\times 2)\times 5^2.\Dih_8$&$2^2.10^2.(\Dih_8^2)$\\
\hline
\end{longtable}

By Proposition \ref{P:tranSg} (resp. Proposition \ref{6GS}), $\Aut(V_{\Lambda_g}^{\hat{g}})$ acts transitively on $\mathcal{S}_g$ if $g\in 4C,6G,8E$ (resp. $g\in6E,10F$).
Hence by Propositions \ref{P:orthg}, the index of $\Aut(V_{\Lambda_g}^{\hat{g}})$ in $O(\irr(V_{\Lambda_g}^{\hat{g}}))$ is equal to that of $C_{\Aut(V_{\Lambda_g})}(\hat{g})/\langle\hat{g}\rangle$ in the stabilizer of an element of $\mathcal{S}_g$ in $O(\irr(V_{\Lambda_g}^{\hat{g}}))$; the index is $2,1,4,2,3$ if $g\in 4C,6E,6G,8E,10F$, respectively;
By Propositions \ref{P:stabVL1} and \ref{P:tranSg}, the stabilizer of an element of $\mathcal{S}_g$ in $\Aut(V_{\Lambda_g}^{\hat{g}})$ is isomorphic to $C_{\Aut(V_{\Lambda_g})}(\hat{g})/\langle \hat{g}\rangle$.

By using MAGMA, we can determine the subgroups $G$ of
$O(\irr(V_{\Lambda_g}^{\hat{g}}) )$ satisfying 
these conditions:
\begin{enumerate}[(1)]
\item the index $|O(\irr(V_{\Lambda_g}^{\hat{g}}) ):G|$ is $2,1,4,2,3$ if $g\in 4C,6E,6G,8E,10F$, respectively;
\item $G$ acts transitively on $\mathcal{S}_g$;
\item the stabilizer of an element in $\mathcal{S}_g$ in $G$ is isomorphic
to that of $C_{\Aut(V_{\Lambda_g})}(\hat{g})/\langle\hat{g}\rangle$ in Table \ref{T:quadVOA}.
\end{enumerate}
It turns out that the subgroups satisfying the above conditions are uniquely determined
up to conjugation in $O(\mathrm{Irr}(V_{\Lambda_g}^{\hat{g}}))$ in each case and  we obtain the following theorem.
\begin{theorem}\label{T:main} Let $g\in O(\Lambda)$.
\begin{enumerate}[{\rm (1)}]
\item If $g$ belongs to the conjugacy class $4C$, then $\aut(V_{\Lambda_g}^{\hat{g}})$ is isomorphic to the index $2$ subgroup of the full orthogonal group 
of $O(\irr(V_{\Lambda_g}^{\hat{g}}))$ and has the shape $$2^{21}.\GO_7(2).$$ 
\item If $g$ belongs to the conjugacy class $6E$, then $\aut(V_{\Lambda_g}^{\hat{g}})$ is isomorphic to the full orthogonal group 
of $O(\irr(V_{\Lambda_g}^{\hat{g}}))$ and has the shape  $$\GO_6^+(2)\times \GO_6^+(3).$$ 
\item If $g$ belongs to the conjugacy class $6G$, then $\aut(V_{\Lambda_g}^{\hat{g}})$ is isomorphic to an index $4$ subgroup of the full orthogonal group 
of $O(\irr(V_{\Lambda_g}^{\hat{g}}))$ and has the shape  
$$ 2^{5+4}.(\Sym_3 \times \Sym_3 \times \Sym_3) \times \Omega_5(3).2.$$
\item If $g$ belongs to the conjugacy class $8E$, then $\aut(V_{\Lambda_g}^{\hat{g}})$ is isomorphic to the index $2$ subgroup of the full orthogonal group of $O(\irr(V_{\Lambda_g}^{\hat{g}}))$ and has the shape $$2^{11+9}.\Sym_6.$$
\item If $g$ belongs to the conjugacy class $10F$, then $\aut(V_{\Lambda_g}^{\hat{g}})$ is isomorphic to an index $3$ subgroup of the full orthogonal group 
of $O(\irr(V_{\Lambda_g}^{\hat{g}}))$ and has the shape $$
(Q_8:2^2).(2\times \Sym_3)\times \GO_4^+(5).$$ 
\end{enumerate}
\end{theorem}

\begin{remark}\label{R:twist} By \eqref{Eq:Sg}, \eqref{Eq:Sg2} and Propositions \ref{P:tranSg} and \ref{6GS}, if the conjugacy class of $g\in O(\Lambda)$ is $4C$, $6E$, $6G$, $8E$ or $10F$, then $V_\Lambda^{\hat{g}}$ has an automorphism which sends $V_{\Lambda_g}(1)$ to an irreducible $V_{\Lambda_g}^{\hat{g}}$-module of twisted type.
\end{remark}

\end{document}